\documentclass[a4paper, reqno]{amsart}

\usepackage{amsmath,amsthm,amssymb,relsize}
\usepackage[nobysame,alphabetic]{amsrefs}
\usepackage[margin=30mm]{geometry} 
\usepackage{mathrsfs,mathtools}

\usepackage{booktabs, multirow, adjustbox, array}

\usepackage{enumerate}

\usepackage[all]{xy}
\usepackage{paralist}

\usepackage{comment}
\usepackage{multirow}

\usepackage{tikz}
\usetikzlibrary{
  cd,
  calc,
  positioning,
  fit,
  arrows,
  decorations.pathreplacing,
  decorations.markings,
  shapes.geometric,
  backgrounds,
  bending
}
\usepackage{tikzsymbols}
\usepackage{xcolor}

\usepackage{tikz-cd}

\definecolor{darkblue}{rgb}{0,0,0.6}

\usepackage[breaklinks, pdftex, ocgcolorlinks,colorlinks=true, citecolor=darkblue, filecolor=darkblue, linkcolor=darkblue, urlcolor=black]{hyperref}
\usepackage[capitalize,noabbrev]{cleveref}

\makeatletter
\newtheorem*{rep@theorem}{\rep@title}
\newcommand{\newreptheorem}[2]{%
\newenvironment{rep#1}[1]{%
 \def\rep@title{#2 \ref{##1}}%
 \begin{rep@theorem}}%
 {\end{rep@theorem}}}
\makeatother

\newtheorem{proposition}{Proposition}[section]
\newtheorem{theorem}[proposition]{Theorem}
\newtheorem*{theorem*}{Theorem}
\newtheorem{corollary}[proposition]{Corollary}
\newtheorem{lemma}[proposition]{Lemma}

\newtheorem{thmx}{Theorem}

\crefname{thmx}{Theorem}{Theorems}

\theoremstyle{definition}
\newtheorem{definition}[proposition]{Definition}
\newtheorem{question}[proposition]{Question}

\theoremstyle{remark}
\newtheorem{remark}[proposition]{Remark}

\newtheorem*{remark*}{Remark}

\newreptheorem{theorem}{Theorem}
\newreptheorem{lemma}{Lemma}
\newreptheorem{proposition}{Proposition}
\newreptheorem{corollary}{Corollary}
\newreptheorem{question}{Question}
\numberwithin{equation}{section}

\newcommand{\R}{\mathbb{R}}

\newcommand{\Z}{\mathbb{Z}}

\newcommand{\wh}{\widehat}
\newcommand{\wt}{\widetilde}
\newcommand{\ol}{\overline}

\newcommand{\M}{\mathcal{M}}

\newcommand{\co}{\colon}

\DeclareMathOperator{\Hom}{Hom}
\DeclareMathOperator{\id}{Id}
\DeclareMathOperator{\Id}{Id}

\DeclareMathOperator{\image}{Im}

\DeclareMathOperator{\im}{Im}

\DeclareMathOperator{\Wh}{Wh}

\DeclareMathOperator{\hAut}{hAut}
\DeclareMathOperator{\hEq}{hEq}

\DeclareMathOperator{\chEq}{chEq}
\DeclareMathOperator{\PD}{PD}
\DeclareMathOperator{\interior}{Int}
\DeclareMathOperator{\hCob}{hCob}

\newcommand{\bsm}{\left(\begin{smallmatrix}}
\newcommand{\esm}{\end{smallmatrix}\right)}
\def\benum{\begin{clist}{(a)}}
\def\eenum{\end{clist}}

\newenvironment{clist}[1]
{\begin{enumerate}[\normalfont #1]}
{\end{enumerate}}

\usepackage{letltxmacro}

\LetLtxMacro\Oldfootnote\footnote

\begin{document}

\title[Generalised doubles and simple homotopy types]{Generalised doubles and simple homotopy types of high dimensional manifolds}

\author{Csaba Nagy}
\address{School of Mathematics and Statistics, University of Glasgow, U.K.}
\email{csaba.nagy@glasgow.ac.uk}

\author{John Nicholson}
\address{School of Mathematics and Statistics, University of Glasgow, U.K.}
\email{john.nicholson@glasgow.ac.uk}

\author{Mark Powell}
\address{School of Mathematics and Statistics, University of Glasgow, U.K.}
\email{mark.powell@glasgow.ac.uk}

\subjclass[2020]{Primary 57N65, 57Q10; Secondary 19J10.  }
\keywords{Manifolds, simple homotopy equivalence.}

\begin{abstract}
We characterise the set of fundamental groups for which there exist $n$-manifolds that are $h$-cobordant (hence homotopy equivalent) but not simple homotopy equivalent, when $n$ is sufficiently large. In particular, for $n \ge 12$ even, we show that examples exist for any finitely presented group $G$ such that the involution on the Whitehead group $\Wh(G)$ is nontrivial. This expands on previous work, where we constructed the first examples of even-dimensional manifolds that are homotopy equivalent but not simple homotopy equivalent. Our construction is based on doubles of thickenings, and a key ingredient of the proof is a formula for the Whitehead torsion of a homotopy equivalence between such manifolds.
\end{abstract}

\maketitle

\vspace{-5mm}

\section{Introduction}

Two finite CW-complexes are said to be \textit{simple homotopy equivalent} if they are related by a sequence of expansions and collapses of cells~\cite{Whitehead-she,Co73}. This gives an equivalence relation which interpolates between homotopy and homeomorphism in the sense that  homeomorphic implies simple homotopy equivalent implies homotopy equivalent \cite{Ch74}.
This notion has proven extremely useful in manifold topology and lies behind the $s$-cobordism theorem, which is the basis for the vast majority of manifold classification results in dimension at least four.

Lens spaces give examples of smooth $n$-manifolds in all odd dimensions $n \ge 3$ that are homotopy equivalent but not simple homotopy equivalent. Examples in all even dimensions $n \ge 4$ were constructed in previous work of the authors \cite{paper1}, with the caveat that the manifolds are topological in dimension four. The examples constructed there are homotopy equivalent to $L \times S^1$ for some lens space $L$, and consequently have fundamental groups $C_\infty \times C_m$ for some $m \ge 2$. 
In this article, we construct new classes of manifolds that are homotopy equivalent but not simple homotopy equivalent.
A key motivation is the following question.  All of our manifolds will be assumed to be smooth, however our results apply equally in the PL or topological categories. 

\begin{question}\label{question-1.1}
For which finitely presented groups $G$ and homomorphisms $w \colon G \to \{\pm 1\}$ are there $n$-manifolds with fundamental group $G$ and orientation character $w$ that are homotopy equivalent but not simple homotopy equivalent?
\end{question}

For a group $G$, Whitehead~\cites{Whitehead-simplicial-spaces,Whitehead-incidence,Whitehead-combinatorial,Whitehead-she} defined the abelian group $\Wh(G)$, and for a homotopy equivalence $f \colon X \to Y$ between finite CW complexes, with $G= \pi_1(Y)$, defined the Whitehead torsion $\tau(f) \in \Wh(G)$. He also proved that a homotopy equivalence $f$ is simple if and only if $\tau(f)=0$. Consequently, for a positive answer to \cref{question-1.1}, it is necessary that $\Wh(G) \neq 0$. As we only consider homotopy equivalences between $n$-manifolds, a stronger condition can be given in terms of $\Wh(G,w)$, the Whitehead group $\Wh(G)$ equipped with the canonical involution $x \mapsto \ol{x}$ induced by the involution on $\Z G$ determined by $g \mapsto w(g)g^{-1}$.
Namely, it is necessary that $\mathcal{J}_n(G,w) \neq 0$ (see \cref{prop:WT-man}), where we use the notation
\[
\begin{aligned}
\mathcal{J}_n(G,w) &:= \left\{ y \in \Wh(G,w) \mid y = -(-1)^n\ol{y} \right\} \leq \Wh(G,w) , \\
\mathcal{I}_n(G,w) &:= \left\{ x - (-1)^n\ol{x} \mid x \in \Wh(G,w) \right\} \leq \mathcal{J}_n(G,w).
\end{aligned}
\]

Our main result is that in high dimensions a slightly stronger condition, the nonvanishing of $\mathcal{I}_n(G,w)$, is sufficient. In fact, this condition precisely characterises when manifolds that are $h$-cobordant but not simple homotopy equivalent exist.

\begin{thmx} \label{thmx:doubles1}
Let $n=9$ or $n \ge 11$, let $G$ be a finitely presented group, and let $w \colon G \rightarrow \left\{ \pm 1 \right\}$ be a homomorphism. Then there is a pair of $n$-manifolds with fundamental group $G$ and orientation character $w$ that are $h$-cobordant but not simple homotopy equivalent if and only if $\mathcal{I}_n(G,w) \neq 0$.
\end{thmx}

Recall that a cobordism $(W;M,N)$ is an \emph{$h$-cobordism} if the inclusion maps $M \to W$ and $N \to W$ are homotopy equivalences. 
In particular, $h$-cobordant but not simple homotopy equivalent manifolds are also examples of homotopy equivalent but not simple homotopy equivalent manifolds.
\cref{thmx:doubles1} will be complemented by \cref{thmx:doubles2}, which gives a sufficient (but possibly not necessary) condition for the existence of manifolds that are homotopy equivalent, but neither simple homotopy equivalent, nor $h$-cobordant. 
The implications of these results for the existence of homotopy equivalent but not simple homotopy equivalent manifolds will be summarised in \cref{ss:group-classification}, and in particular in \cref{table:q2}.

\cref{thmx:doubles1} can be used to construct examples of manifolds that are homotopy but not simple homotopy equivalent, provided that $\mathcal{I}_n(G,w) \ne 0$. If $G$ is torsion free and satisfies the Farrell Jones conjecture, then $\Wh(G)=0$ and so $\mathcal{I}_n(G,w)=0$ for all orientation characters $w$. 
We therefore consider groups $G$ that have nontrivial torsion.
It was shown by Wall \cite{Wa74} that, if $G$ is a finite group and $w=1$, then the involution is trivial when restricted to $\Wh'(G)$, the free part of $\Wh(G)$. This can be used to show that the involution is trivial for a range of finite groups, including finite abelian groups \cite[Theorem 1]{Ba77} and finite $3$-manifold groups \cite[p742]{KS92}.

If $n$ is even, then $\mathcal{I}_n(G,w) \ne 0$ if and only if the involution on $\Wh(G,w)$ is nontrivial.
It is not straightforward to find groups with this property and this explains why it was so much harder to find examples in even dimensions.
Nonetheless the involution has been shown to be nontrivial for groups of the form $C_{\infty} \times C_m$ for certain $m \ge 2$ (see \cite[Theorem 5.15 (i)]{paper1} and \cite[p.~421]{Mi66}) and for certain finite $p$-groups (see \cite[Proposition 24]{Ol80}, \cite[Example 8.11]{Ol88}), the smallest of which has order $3^5=243$.
Using the latter examples, we now obtain the following.

\begin{corollary}\label{cor:finite-pi1-examples}
    Let $n \ge 12$ be even. Then there exist orientable $n$-manifolds with finite fundamental group that are homotopy equivalent but not simple homotopy equivalent. 
\end{corollary}

As mentioned above, lens spaces give examples in all odd dimensions $n \ge 3$ since their fundamental groups are finite cyclic.
\cref{cor:finite-pi1-examples} demonstrates the utility of \cref{thmx:doubles1} and sits in contrast to the examples constructed in \cite{paper1}, which were all homotopy equivalent to $L \times S^1$ for some lens space $L$ and so had fundamental groups of the form $C_\infty \times C_m$.

\subsection{Simple homotopy manifold sets} \label{ss:shms}

To study the difference between homotopy equivalence and simple homotopy equivalence, we introduced the \emph{simple homotopy manifold set} of a closed $n$-manifold $M$ in \cite{paper1}. This is the set of $n$-manifolds homotopy equivalent to $M$ up to simple homotopy equivalence: 
\[
\M^h_s(M) \coloneqq \{ \text{$n$-manifolds } N \mid N \simeq M \} / \simeq_s
\]
where we write~$\simeq$ for homotopy equivalence and ~$\simeq_s$ for simple homotopy equivalence. Hence there is a pair of manifolds (with a given fundamental group and orientation character) that are homotopy equivalent but not simple homotopy equivalent if and only if there is an $M$ with $|\M^h_s(M)| > 1$. 

To understand $\M^h_s(M)$, it is helpful to also consider the following variations:
\[
\begin{aligned}
\M^{\hCob}_s(M) & \coloneqq  \left\{ \text{$n$-manifolds } N \mid N \text{ is $h$-cobordant to } M \right\} / \simeq_s \\
\M^h_{s,\hCob}(M) & \coloneqq  \left\{ \text{$n$-manifolds } N \mid N \simeq M \right\} / \langle \simeq_s, \hCob \rangle
\end{aligned}
\]
where $\langle \simeq_s, \hCob \rangle$ denotes the equivalence relation generated by simple homotopy equivalence and $h$-cobordism (see \cite[Theorem E]{paper1}). In particular, if $|\M^{\hCob}_s(M)| > 1$ or $|\M^h_{s,\hCob}(M)| > 1$, then $|\M^h_s(M)| > 1$. By \cite[Proposition 4.22]{paper1} the converse also holds when $n \geq 5$. 

To state the results about simple homotopy manifold sets that we will need, we recall that the Tate cohomology group $\wh{H}^{n+1}(C_2;\Wh(G,w))$ is canonically identified with $\mathcal{J}_n(G,w) / \mathcal{I}_n(G,w)$, and we denote the quotient map by 
\[
\pi \colon \mathcal{J}_n(G,w) \rightarrow \wh{H}^{n+1}(C_2;\Wh(G,w)) .
\]
We will also make use of the homomorphism $\psi \colon L_{n+1}^h(\Z G, w) \rightarrow \wh{H}^{n+1}(C_2;\Wh(G,w))$ from the Ranicki--Rothenberg exact sequence (see \cite{Shaneson-GxZ}, \cite[\S9]{Ranicki-80-I}). Finally, for an $n$-manifold $M$ with fundamental group $G$ and orientation character $w \colon G \rightarrow \left\{ \pm 1 \right\}$, we introduce the following notation: 
\[
\begin{aligned}
T(M) &\coloneqq  \left\{ \tau(g) \mid g \in \hAut(M) \right\} \subseteq \mathcal{J}_n(G,w) \\
U(M) &\coloneqq \left\{ \pi(\tau(g)) \mid g \in \hAut(M) \right\} \subseteq \wh{H}^{n+1}(C_2;\Wh(G,w)) 
\end{aligned}
\]
where $\hAut(M)$ is the group of homotopy automorphisms of $M$ (and we used \cref{prop:WT-man}).

The following result, which was proved in \cite[Propositions 4.20--21]{paper1}, will be our main tool for determining when a simple homotopy manifold set is nontrivial or trivial.

\begin{proposition} \label{prop:TU-cond}
Let $n \geq 5$ and let $M$ be an $n$-manifold. Let $G = \pi_1(M)$ with orientation character $w \colon G \rightarrow \left\{ \pm 1 \right\}$. Then
\benum
\item\label{prop:TU-cond-a} $|\M^{\hCob}_s(M)| > 1$ if and only if $\mathcal{I}_n(G,w) \setminus T(M)$ is nonempty.
\item\label{prop:TU-cond-b} If $\image(\psi) \setminus U(M)$ is nonempty, then $|\M^h_{s,\hCob}(M)| > 1$. 
\item\label{prop:TU-cond-c} If $|\M^h_{s,\hCob}(M)| > 1$, then $\wh{H}^{n+1}(C_2;\Wh(G,w)) \neq 0$. 
\eenum
\end{proposition}

\subsection{Homotopy equivalences of doubles} \label{ss:intro:doubles}

By \cref{prop:TU-cond}, to show that a simple homotopy manifold set is nontrivial for some $M$, we need to understand the involution on the Whitehead group of $\pi_1(M)$, and the group of homotopy automorphisms $\hAut(M)$. In this paper we will be focusing on the latter. Our aim is to develop a systematic method for constructing manifolds $M$, with arbitrary predetermined fundamental group $G$ and orientation character $w$, such that we have control over the Whitehead torsion of homotopy automorphisms of $M$. To achieve this, we will study doubles and homotopy equivalences between them. 

Our central result is a formula for the Whitehead torsion of such a homotopy equivalence (\cref{theorem:SP-manifolds-and-tau-intro}), which we will discuss below. By applying this formula to certain doubles, we get sufficient conditions for the existence of an $M$ with nontrivial $\mathcal{M}^{\hCob}_s(M)$ or $\mathcal{M}^h_{s,\hCob}(M)$ (and hence $\mathcal{M}^h_s(M)$), expressed in terms of the involution on $\Wh(G,w)$. These, and a few further applications of \cref{theorem:SP-manifolds-and-tau-intro}, will be discussed in \cref{ss:appl}.

We start by introducing the doubles that we will study. Fix positive integers $n, k$ such that $n \geq \max(6,2k+2)$. Let $K$ be a $k$-complex, i.e.\ a finite $k$-dimensional CW complex, and let $T$ be an $n$-dimensional thickening of $K$, i.e.\ an $n$-manifold with boundary together with a simple homotopy equivalence $f_T \colon K \rightarrow T$ (see \cref{s:thick}). Then a \textit{double} over $K$ is a manifold of the form $M = T \cup_{\id_{\partial T}} T$ (a trivial double), $M = T \cup_g T$ where $g \colon \partial T \rightarrow \partial T$ is a diffeomorphism (a twisted double), or $M = T \cup W \cup T$ where $W$ is an $h$-cobordism between two copies of $\partial T$ (a generalised double). In each case $M$ comes equipped with a canonical map $\varphi \colon K \rightarrow M$, which is the composition of $f_T$ and the inclusion of the first component $T \rightarrow M$. 

If $M$ is an $n$-manifold and $K$ is a $k$-complex, then we will call a map $\varphi \colon K \rightarrow M$ a \emph{polarisation} of $M$, and the pair $(M,\varphi)$ a \emph{polarised manifold}. We will say that $(M,\varphi)$ has a \emph{trivial/twisted/generalised double structure} if, for some thickening $T$ of $K$, $M$ has a decomposition as above such that the canonical map $K \rightarrow M$ is homotopic to $\varphi$. 

We will obtain the following recognition criterion for generalised doubles (see \cref{prop:double-conn}).

\begin{proposition} \label{prop:double-conn-intro}
Let $(M,\varphi)$ be a polarised manifold. Then $(M,\varphi)$ has a generalised double structure if and only if $\varphi$ is $\left\lfloor \frac{n}{2} \right\rfloor$-connected.
\end{proposition}

It follows that the class of generalised doubles is closed under homotopy equivalence, i.e.\ if $(M,\varphi)$ has a generalised double structure and $M \simeq N$, then $(N,\psi)$ also has a generalised double structure for some $\psi$. Note that the same is not true for the classes of twisted and trivial doubles. 

Given a polarised manifold $(M,\varphi)$ such that $\varphi$ is $\left\lfloor \frac{n}{2} \right\rfloor$-connected, we can define an invariant $\tau(M,\varphi)$, called the \emph{Whitehead torsion of $(M,\varphi)$}, as follows. There is a thickening $f_T \colon K \rightarrow T$ of~$K$ and an embedding $i \colon T \rightarrow M$ such that $i \circ f_T \simeq \varphi$. Let $C = M \setminus i(\interior T)$, then by general position and because $n \geq 2k+2$, the map $\varphi$ is homotopic to a map that factors as $K \to C \to M$, with the first map a homotopy equivalence $\varphi' \colon K \rightarrow C$. In particular the inclusion $C \rightarrow M$ induces an isomorphism $\pi_1(C) \cong \pi_1(M)$, so we can identify $\Wh(\pi_1(C))$ with $\Wh(\pi_1(M))$.

\begin{definition}
Let $\tau(M,\varphi) = \tau(\varphi') \in \Wh(\pi_1(M))$.
\end{definition}

This invariant will show up in the formula of \cref{theorem:SP-manifolds-and-tau-intro} as an error term. We will also prove in \cref{cor:tau-triv-double} that if $(M,\varphi)$ has a trivial double structure, then $\tau(M,\varphi)=0$. The converse does not hold, but if $\tau(M,\varphi)=0$, then $(M,\varphi)$ has a twisted double structure, and more generally see \cref{prop:characterises-twisted-double-structures} for a result characterising precisely when a twisted double structure exists.

Next we consider homotopy equivalences between doubles. If $(M,\varphi)$ is a polarised manifold and $\varphi$ is $\left\lfloor \frac{n}{2} \right\rfloor$-connected, then we can find a CW decomposition of $M$ such that $\varphi$ is the embedding of its $k$-skeleton. Suppose that $N$ is another $n$-manifold with an $\left\lfloor \frac{n}{2} \right\rfloor$-connected polarisation $\psi \colon L \rightarrow N$ for a $k$-dimensional CW complex $L$. Then by cellular approximation any homotopy equivalence $f \colon M \rightarrow N$ restricts to a map $\alpha \colon K \rightarrow L$. This $\alpha$ is also a homotopy equivalence if $(M,\varphi)$ and $(N,\psi)$ satisfy some mild restrictions on their dimensions or double structures; in this case we call them \emph{split polarised} (SP) manifolds (see \cref{def:SP}). The following key theorem then allows us to compute the Whitehead torsion of $f$ from that of $\alpha$ (see \cref{theorem:heq-diag}). 

\begin{thmx}\label{theorem:SP-manifolds-and-tau-intro}
Suppose that $(M,\varphi)$ and $(N,\psi)$ are split polarised manifolds and $f \colon M \rightarrow N$ is a homotopy equivalence. Then
\[
\tau(f) = \tau(\alpha) - (-1)^n \overline{\tau(\alpha)} + \tau(N,\psi) - f_*(\tau(M,\varphi)) \in \Wh(\pi_1(N),w_N)
\]
where $\alpha \colon K \rightarrow L$ is the restriction of $f$, $\Wh(\pi_1(L))$ is identified with $\Wh(\pi_1(N))$ via $\psi_*$ and $w_N$ is the orientation character of $N$.
\end{thmx}

To prove \cref{theorem:SP-manifolds-and-tau-intro}, we start by considering the algebraic version of the problem. We consider \emph{split chain complexes}, which are those that split as the direct sum of their lower and upper halves (see \cref{def:lu,def:split}). Provided that they are also based and satisfy a version of Poincar\'e duality, we compute the Whitehead torsion of a chain homotopy equivalence between two such chain complexes in terms of the Whitehead torsion of the restriction of the chain map to the lower half (see \cref{lem:cheq-diag}). The formula contains two error terms, which are intrinsic to the two chain complexes, namely they are given by the Whitehead torsion of the restriction of the Poincar\'e duality map. To derive \cref{theorem:SP-manifolds-and-tau-intro} from this, we show that if $(M,\varphi)$ is an SP manifold, then its cellular chain complex splits and its Whitehead torsion $\tau(M,\varphi)$ is equal to the associated error term (see \cref{theorem:tau-pd}). 

\subsection{Applications of \cref{theorem:SP-manifolds-and-tau-intro}} \label{ss:appl}

In the main applications of \cref{theorem:SP-manifolds-and-tau-intro} we restrict to special types of doubles where we have control over the right hand side of the formula. In particular, when $\tau(M,\varphi)=0$ (which can be ensured by taking $M$ to be a trivial double), then we have $\tau(f) \in \mathcal{I}_n(G,w)$ for every homotopy automorphism $f$ of $M$. By \cref{prop:TU-cond} \eqref{prop:TU-cond-b} we have the following (see \cref{thm:hausmann}). 

\begin{theorem} \label{thmx:doubles2}
Let $n \geq 5$, let $G$ be a finitely presented group and let $w \colon G \rightarrow \left\{ \pm 1 \right\}$ be such that $\psi \colon L_{n+1}^h(\Z G, w) \rightarrow \wh{H}^{n+1}(C_2;\Wh(G,w))$ is nontrivial. Then there exists an $n$-manifold $M$ with fundamental group $G$ and orientation character $w$ such that $|\M^h_{s,\hCob}(M)|>1$.
\end{theorem}

\begin{remark} (a)
For $n \geq 6$, this can also be deduced from results in Hausmann's unpublished preprint \cite[Sections 9--10]{hausmann-mwmdh}. 

(b)
For $G=C_{\infty} \times C_m$, $w=1$, and $n$ even, the map $\psi$ can be nontrivial by combining \cite[Proposition 3.12]{paper1} with \cite[Theorem 1.8 (ii)]{paper1}. More specifically, by \cite[Proposition 11.16]{paper1},  $\psi \ne 0$ for infinitely many groups of the form $C_\infty \times C_{2^k}$.
\end{remark}

When we also have some control over the possible values of $\tau(\alpha)$, \cref{theorem:SP-manifolds-and-tau-intro} gives even stricter restrictions on the Whitehead torsions of homotopy equivalences and automorphisms. For example, if $K$ and $L$ are $k$-manifolds with orientation character $w$, then $\tau(\alpha) \in \mathcal{J}_k(G,w)$ (see \cref{prop:WT-man}). We obtain the following simple homotopy rigidity theorem for sphere bundles.

\begin{thmx} \label{thmx:sb}
Suppose that $j > k$ are positive integers and $j$ is odd. Let $K$ and $L$ be $k$-manifolds, and let $S^j \to M \to K$ and $S^j \to N \to L$ be orientable $($linear$)$ sphere bundles. Then every homotopy equivalence $f \co M \to N$ is simple.
\end{thmx}

\begin{remark}
If we think of the Whitehead torsion as an invariant analogous to the Euler characteristic (cf.\ \cite[Lemma 2.18]{paper1}), then \cref{thmx:sb} can be regarded as the analogue of the fact that odd dimensional manifolds have vanishing Euler characteristic.
\end{remark}

Combined with \cref{prop:TU-cond} \eqref{prop:TU-cond-a}, \cref{thmx:sb} leads to a proof of \cref{thmx:doubles1} (see \cref{ss:doubles-man}).  At the same time, it produces a class of manifolds, with arbitrary fundamental groups, within which two manifolds are homotopy equivalent if and only if they are simple homotopy equivalent.  

To get more concrete examples, next we consider doubles over certain $2$-complexes $X$ with fundamental group $C_{\infty} \times C_m$, for which Metzler showed that $\tau \colon \hAut(X) \to \Wh(C_{\infty} \times C_m)$ is not surjective \cite[Theorem 1]{Me79}. Using these complexes $X$, improved constraints on the set $\tau(\hAut(X))$, and \cref{prop:TU-cond} \eqref{prop:TU-cond-a}, we obtain the following theorem. 

\begin{theorem} \label{thmx:metzler}
Let $n \geq 5$ and let $m \ge 2$ be such that $\{x-(-1)^n \ol{x} \mid x \in \wt K_0(\Z C_m)\} \ne 0$. Then there is an orientable $n$-manifold $M$ with fundamental group $C_{\infty} \times C_m$ such that $|\M^{\hCob}_s(M)| > 1$.
\end{theorem}

\begin{remark}
If $n$ is even, then examples where $\{x- \ol{x} \mid x \in \wt K_0(\Z C_m)\} \ne 0$ are given in \cite[Theorem 1.7 (i)]{paper1}. If $n$ is odd then, by \cite[Lemma 10.2 (i)]{paper1} and a similar argument to the one used in \cite[Lemma 11.4]{paper1}, we have
$|\{x + \ol{x} \mid x \in \wt K_0(\Z C_m)\}| \ge h_m^{+}$ where $h_m^{+}$ denotes the plus part of the $m$th cyclotomic class number. For example, we have that $h_{136}^+ \ne 1$ \cite[p.~421]{Wa97}. 
\end{remark}

Finally we note that $\tau(M, \varphi)$ can be used to define an invariant for certain unpolarised manifolds. We say that $M$ is a \emph{split manifold} if there exists  a polarisation $\varphi \colon K \rightarrow M$ such that $(M, \varphi)$ is an SP manifold. For such an $M$ we define
\[
\tau(M) := \pi(\tau(M,\varphi)) \in \wh{H}^{n+1}(C_2;\Wh(\pi_1(M),w))
\]
where $w \colon \pi_1(M) \rightarrow \left\{ \pm 1 \right\}$ is the orientation character of $M$. It can be shown using \cref{theorem:SP-manifolds-and-tau-intro} that this is well-defined (see \cref{ss:depend}). If a split manifold $M$ is also a manifold without middle dimensional handles in the sense of Hausmann \cite{hausmann-mwmdh}, then $\tau(M)$ recovers the ``torsion invariant" defined in \cite[Section 9]{hausmann-mwmdh}, which was shown to be invariant under simple homotopy equivalences and homotopy equivalences induced by $h$-cobordisms. We prove the following (see \cref{theorem:tau-complete}).  

\begin{theorem} \label{thm:i-compl}
On the class of split manifolds, $\tau(M)$ is a complete invariant for the equivalence relation generated by simple homotopy equivalence and $h$-cobordism.
\end{theorem}

\subsection{Summary} \label{ss:group-classification}

Let $G$ be a finitely presented group with a homomorphism $w \colon G \to \{\pm 1\}$. We now give a brief overview of how the above results can be used to decide whether a pair of homotopy equivalent but not simple homotopy equivalent $n$-manifolds with fundamental group $G$ and orientation character $w$ exists -- equivalently, an whether an $M$ exists with 1-type $(G,w)$ and $|\mathcal{M}^h_s(M)| > 1$. We also consider the analogous questions for $\M^{\hCob}_s$ and $\M^h_{s,\hCob}$.

First, \cref{thmx:doubles1} characterises pairs $(G,w)$ such that $|\mathcal{M}^{\hCob}_s(M)| > 1$ for some $M$ with fundamental group $G$ and orientation character $w$ in terms of the (non-)vanishing of $\mathcal{I}_n(G,w)$. 
For $\mathcal{M}^h_{s,\hCob}$, \cref{prop:TU-cond} \eqref{prop:TU-cond-c} and \cref{thmx:doubles2} show that the nontriviality of $\wh H^{n+1}(C_2;\Wh(G,w))$ is necessary and the nonvanishing of $\psi = \psi^{n+1}_{(G,w)} \colon L_{n+1}^h(\Z G, w) \rightarrow \wh{H}^{n+1}(C_2;\Wh(G,w))$ is sufficient. The case when $\wh H^{n+1}(C_2;\Wh(G,w)) \neq 0$ but $\psi^{n+1}_{(G,w)} = 0$ is open in general.

Finally, if at least one of $\mathcal{M}^{\hCob}_s(M)$ and $\mathcal{M}^h_{s,\hCob}(M)$ is nontrivial, then $\mathcal{M}^h_s(M)$ is nontrivial as well. On the other hand, if $\mathcal{J}_n(G,w) = 0$ (equivalently, $\mathcal{I}_n(G,w) = \wh{H}^{n+1}(C_2;\Wh(G,w)) = 0$), then by \cref{prop:WT-man} $\mathcal{M}^h_s(M)$ is trivial (and hence $\mathcal{M}^{\hCob}_s(M)$ and $\mathcal{M}^h_{s,\hCob}(M)$ are trivial too). Thus the existence of an $M$ with $|\mathcal{M}^h_s(M)| > 1$ can be decided in all cases except when $\mathcal{I}_n(G,w) = 0$, $\wh H^{n+1}(C_2;\Wh(G,w)) \neq 0$ and $\psi^{n+1}_{(G,w)} = 0$. We note that this case is nonempty: for example, for $n$ even,  the group $C_4 \times C_4$, with $w=1$, falls into this category. Details will be postponed  for future work. 

These results are summarised in \cref{table:q2}, where we also indicate the restrictions on $n$ needed for our theorems to apply. (We are not asserting that the dimensional ranges given are optimal.)

\begin{table}[h] 
\begin{center}
\begin{tabular}{|c|c|c||c|c|c|c|c|c|}
\hline
\!$\mathcal{I}_n(G,w)$\! & \!\!$\wh H^{n+1}(C_2;\Wh(G,w))$\!\! & \!$\psi^{n+1}_{(G,w)}$\! & \multicolumn{2}{c|}{$\mathcal{M}^h_s$} & \multicolumn{2}{c|}{$\mathcal{M}^{\hCob}_s$} & \multicolumn{2}{c|}{$\mathcal{M}^h_{s,\hCob}$} \\
\hline
$= 0$ & $=0$ & $=0$ & No & all $n$ & No & all $n$ & No & all $n$ \\
\hline
$= 0$ & $\ne 0$ & $\ne 0$ & Yes & $n \geq 5$ & No & all $n$ & Yes & $n \geq 5$ \\
\hline
$= 0$ & $\ne 0$ & $=0$ & Open & - & No & all $n$ & Open & - \\
\hline
$\ne 0$ & $=0$ & $=0$ & Yes & $n = 9, \geq 11$ & Yes & $n = 9, \geq 11$ & No & $n \geq 5$ \\
\hline
$\ne 0$ & $\ne 0$ & $\ne 0$ & Yes & $n \geq 5$ & Yes & $n = 9, \geq 11$ & Yes & $n \geq 5$ \\
\hline
$\ne 0$ & $\ne 0$ & $=0$ & Yes & $n = 9, \geq 11$ & Yes & $n = 9, \geq 11$ & Open & - \\
\hline
\end{tabular}
\vspace{1mm}
\end{center}
\caption{Is there an $n$-manifold $M$ with fundamental group $G$ and orientation character $w$ such that $|\mathcal{M}^h_s(M)| > 1$ (resp.\ $|\mathcal{M}^{\hCob}_s(M)| > 1$, $|\mathcal{M}^h_{s,\hCob}(M)| > 1$)?}
\label{table:classifying-he-vs-she}
\vspace{-5mm}
\label{table:q2}
\end{table}

\begin{remark}
The results listed in \cref{table:q2} also apply in the category of topological or PL manifolds. That is because when there are examples of smooth manifolds that are, for instance, homotopy equivalent but not simple homotopy equivalent, then those are also examples in the other categories. The negative results in the first row and the $\mathcal{M}^{\hCob}_s$ column rely on \cref{prop:WT-man,prop:WT-hcob}, which only use Poincar\'e duality. Finally, when $\wh H^{n+1}(C_2;\Wh(G,w))=0$ but $\mathcal{I}_n(G,w) = \mathcal{J}_n(G,w) \neq 0$, the triviality of $\mathcal{M}^h_{s,\hCob}$ follows ultimately from \cite[Proposition 4.21]{paper1}, so it holds in all three categories when $n \geq 5$ (and also for topological $4$-manifolds with good fundamental group). 
\end{remark}

\subsection*{Organisation of the paper}

In \cref{s:prelims} we will briefly recall the basic preliminaries on Whitehead torsion, taken from \cite{paper1}.
In \cref{s:scc} we consider split chain complexes, and prove the algebraic version of \cref{theorem:SP-manifolds-and-tau-intro}. In \cref{s:thick}, we recall Wall's results on thickenings. In \cref{s:gd}, we introduce polarised doubles, SP manifolds and the invariant $\tau(M,\varphi)$, and then identify $\tau(M,\varphi)$ with an invariant of the cellular chain complex of $M$. In \cref{s:doubles-main}, we combine the previous results to prove \cref{theorem:SP-manifolds-and-tau-intro}. In \cref{s:appl}, we consider various applications of \cref{theorem:SP-manifolds-and-tau-intro}, including Theorems \ref{thmx:doubles1}, \ref{thmx:sb}, \ref{thmx:doubles2}, \ref{thmx:metzler} and \ref{thm:i-compl}. 

\subsection*{Acknowledgements}

CsN was supported by EPSRC New Investigator grant EP/T028335/2.
JN was supported by the Heilbronn Institute for Mathematical Research and a Rankin-Sneddon Research Fellowship from the University of Glasgow.
MP was partially supported by EPSRC New Investigator grant EP/T028335/2 and EPSRC New Horizons grant EP/V04821X/2.

\section{Whitehead torsion} \label{s:prelims}

We will use the conventions established in \cite[Section 2]{paper1}. The main sources are Milnor~\cite{Mi66}, Cohen~\cite{Co73}, and Davis-Kirk~\cite{DK01}. We will assume familiarity with the definition of simple homotopy equivalence $\simeq_s$, the Whitehead group $\Wh(G)$ of a group $G$, and the Whitehead torsion $\tau(f) \in \Wh(G)$ of a chain homotopy equivalence $f : C_* \to D_*$ between chain complexes of finitely generated, free, based (left) $\Z G$-modules (see \cite[Section 2]{paper1}). For the reader's convenience, we briefly recall the main properties that we will need.

\begin{proposition}[{\cite[Theorem 11.27]{DK01}}] \label{prop:chain-hom-chain-equivs-same-torsion}
Let $f,g \colon C_* \to D_*$ be homotopic chain homotopy equivalences. Then $\tau(f) = \tau(g)$.
\end{proposition}

\begin{lemma}[{\cite[Theorem 11.28]{DK01}}] \label{lem:WT-ch-comp}
Let $f \colon C_* \rightarrow D_*$ and $g \colon D_* \rightarrow E_*$ be chain homotopy equivalences. Then $\tau(g \circ f) = \tau(f) + \tau(g)$. In particular $\tau(\Id)=0$. 
\end{lemma}

\begin{lemma}[{\cite[Lemma~2.17]{paper1}}] \label{lem:WT-ch-add}
Let $0 \rightarrow C'_* \rightarrow C_* \rightarrow C''_* \rightarrow 0$ and $0 \rightarrow D'_* \rightarrow D_* \rightarrow D''_* \rightarrow 0$ be based short exact sequences of chain complexes, and let $(f',f,f'')$ be a morphism between them, where $f' \colon C'_* \rightarrow D'_*$, $f \colon C_* \rightarrow D_*$, and $f'' \colon C''_* \rightarrow D''_*$ are chain homotopy equivalences. Then $\tau(f) = \tau(f') + \tau(f'')$. 
\end{lemma}

\begin{lemma}[{\cite[Lemma~2.19]{paper1}}] \label{lem:WT-ch-shift}
Let $f \colon C_* \rightarrow D_*$ be a chain homotopy equivalence. For every $k \in \Z$, it can also be regarded as a chain homotopy equivalence $f \colon C_{k+*} \rightarrow D_{k+*}$, and we have $\tau(f \colon C_{k+*} \rightarrow D_{k+*}) = (-1)^k \tau(f \colon C_* \rightarrow D_*)$.
\end{lemma}

We can also define the Whitehead torsion for a map between cochain complexes. 

\begin{definition} \label{def:WT-ch-cochain}
Let $f \colon C^* \rightarrow D^*$ be a homotopy equivalence of cochain complexes of finitely generated, free, based, left $\Z G$-modules. It can be regarded as a homotopy equivalence of chain complexes $f \colon C^{-*} \rightarrow D^{-*}$, and we define $\tau(f \colon C^* \rightarrow D^*)  \coloneqq  \tau(f \colon C^{-*} \rightarrow D^{-*})$.
\end{definition}

If a group $G$ is equipped with a group homomorphism $w \colon G \to \{\pm 1\}$, then $w$ determines an involution on the group ring $\Z G$. So if  $C_*$ is a finitely generated, free, based, left $\Z G$-module chain complex, then we can define the dual cochain complex $C^*$, which also consists of finitely generated, free, based, left $\Z G$-modules. Moreover, the involution on the group ring induces an involution $x \mapsto \overline{x}$ on the Whitehead group $\Wh(G,w)$, and we have the following.

\begin{lemma}[{\cite[Lemma~2.24]{paper1}}] \label{lem:WT-ch-dual}
Let $f \colon C_* \rightarrow D_*$ be a chain homotopy equivalence and let $f^* \colon D^* \rightarrow C^*$ be its dual. Then $\tau(f^*) = \ol{\tau(f)} \in \Wh(G,w)$. 
\end{lemma}

The following notation will be used when we need to change the underlying (group) ring of a module or chain complex.

\begin{definition} \label{def:twist-module}
Let $A$ and $B$ be groups, $X$ a left $\Z B$-module and $\theta \in \Hom(A,B)$. The left $\Z A$-module $X_{\theta}$ is defined as follows. The underlying abelian group of $X_{\theta}$ is the same as that of $X$. For every $a \in A$ and $x \in X_{\theta}$ let $ax = \theta(a) \cdot x$, where $\cdot$ denotes multiplication in $X$.

Similarly, if $Y$ is a right $\Z B$-module, then the right $\Z A$-module $Y^{\theta}$ is equal to $Y$ as an abelian group, and $ya = y \cdot \theta(a)$ for every $a \in A$ and $y \in Y^{\theta}$, where $\cdot$ denotes multiplication in $Y$.
\end{definition}

Let $X$ and $Y$ be finite CW complexes, and let $F := \pi_1(X)$ and $G := \pi_1(Y)$. The cellular chain complex of $Y$ with $\Z G$ coefficients is $C_*(Y; \Z G)$, which is a finitely generated, free, left $\Z G$-module chain complex, and similarly $C_*(X; \Z F)$ is a $\Z F$-module chain complex. Let $f \colon X \to Y$ be a homotopy equivalence, and let $\theta = \pi_1(f) \colon F \rightarrow G$. The right $\Z G$-module $\Z G$ corresponds to a local coefficient system on $Y$, which is pulled back to the local coefficient system on $X$ corresponding to the right $\Z F$-module $\Z G^{\theta}$. Therefore (after cellular approximation) $f$ induces a chain homotopy equivalence $f_* \colon C_*(X; \Z G^{\theta}) \rightarrow C_*(Y; \Z G)$ of left $\Z G$-module chain complexes, and the Whitehead torsion of $f$ is defined to be the Whitehead torsion of $f_*$. We will use the following properties.

\begin{theorem}[Chapman~\cite{Ch74}]\label{thm:chapman}
 Let $f \colon X \to Y$ be a homeomorphism between compact, connected CW complexes. Then $f$ is a simple homotopy equivalence. 
\end{theorem}

\begin{proposition}[{\cite[Statement~22.4]{Co73}}] \label{prop:WT-composition}
Let $X$, $Y$, and $Z$ be finite CW complexes and let $f \colon X \to Y$ and $g \colon Y \to Z$ be homotopy equivalences. Then $\tau(g \circ f) = \tau(g) + g_*(\tau(f))$.
\end{proposition}

The involution of the Whitehead group and the subgroups $\mathcal{J}_n(G,w)$ and $\mathcal{I}_n(G,w)$ play an important role when considering homotopy equivalences between manifolds, due to the following two results, both of which are consequences of Poincar\'e duality (cf.\ \cite{Wall-SCM}, \cite{Kirby-Siebenmann:1977-1}, \cite{Mi66}).

\begin{proposition}[{\cite[Proposition~2.35]{paper1}}] \label{prop:WT-man}
Let $M$ and $N$ be $n$-manifolds. Let $G = \pi_1(N)$ with orientation character $w \colon G \rightarrow \left\{ \pm 1 \right\}$, and let $f \colon M \to N$ be a homotopy equivalence. Then $\tau(f) \in \mathcal{J}_n(G,w)$. 
\end{proposition}

\begin{proposition}[{\cite[Proposition~2.38]{paper1}}] \label{prop:WT-hcob}
Let $W$ be an $h$-cobordism between $n$-manifolds $M$ and $N$. Let $G = \pi_1(W)$, let $w \colon G \rightarrow \left\{ \pm 1 \right\}$ be the orientation character of $W$, and identify $\pi_1(M)$ with $G$ via the inclusion. If $f \colon N \rightarrow M$ denotes the homotopy equivalence induced by $W$, then $\tau(f) = -\tau(W,M) + (-1)^n \overline{\tau(W,M)} \in \mathcal{I}_n(G,w)$.
\end{proposition}

\section{Split chain complexes} \label{s:scc}

In this section we define split chain complexes and prove \cref{lem:cheq-diag} about computing the Whitehead torsion of a chain homotopy equivalence between two split chain complexes.

Fix some positive integer $n$ and a group $G$ with orientation character $w\colon G \rightarrow \left\{ \pm 1 \right\}$. Every chain complex $C_*$ will be assumed to consist of finitely generated free, left $\Z G$-modules, satisfying the condition $C_i=0$ for $i<0$ and $i>n$ (and similarly for cochain complexes).

\begin{definition} \label{def:lu}
Let $(C_*,d_*)$ be a chain complex. We define the chain complex $(C^{\ell}_*,d^{\ell}_*)$ by $C^{\ell}_i = C_i$ and $d^{\ell}_i = d_i \colon C_i \to C_{i-1}$ if $i < \frac{n}{2}$ and $C^{\ell}_i=0$ if $i \geq \frac{n}{2}$. Similarly, we define $(C^u_*,d^u_*)$ by $C^u_i = C_i$ if $i > \frac{n}{2}$, $d^u_i = d_i$ if $i > \frac{n}{2} + 1$, and $C^u_i=0$ if $i \leq \frac{n}{2}$. 

Let $(C^*,d^*)$ be a cochain complex. We define $((C^u)^*,(d^u)^*)$ by $(C^u)^i = C^i$ and $(d^u)^i = d^i \colon C^i \to C^{i+1}$ if $i > \frac{n}{2}$ and $(C^u)^i=0$ if $i \leq \frac{n}{2}$. Similarly, we define $((C^{\ell})^*,(d^{\ell})^*)$ by $(C^{\ell})^i = C^i$ if $i < \frac{n}{2}$, $(d^{\ell})^i = d^i$ if $i < \frac{n}{2} - 1$, and $(C^{\ell})^i=0$ if $i \geq \frac{n}{2}$. 
\end{definition}

\begin{definition} \label{def:split}
We say that a chain complex $(C_*,d_*)$ \emph{splits} if $C_{n/2}=0$ (if $n$ is even) or $d_{\left\lceil n/2 \right\rceil}=0$ (if $n$ is odd). 
We say that a cochain complex $(C^*,d^*)$ \emph{splits} if $C^{n/2}=0$ (if $n$ is even) or $d^{\left\lfloor n/2 \right\rfloor}=0$ (if $n$ is odd). 
\end{definition}

Note that $C_*$ (resp.\ $C^*$) splits if and only if $C_* \cong C^{\ell}_* \oplus C^u_*$ (resp.\ $C^* \cong (C^u)^* \oplus (C^{\ell})^*$). 

\begin{definition}
Let $C_*$ and $D_*$ be chain complexes. We define $\chEq(C_*,D_*) \subseteq \Hom(C_*,D_*) / {\simeq}$ to be the set of chain homotopy classes of chain homotopy equivalences $C_* \rightarrow D_*$.
\end{definition}

\begin{lemma} \label{lem:split-cheq}
Suppose that $(C_*,d^C_*)$ and $(D_*d^D_*)$ are split chain complexes. Then there is a bijection 
\[\chEq(C_*,D_*) \cong \chEq(C^{\ell}_*,D^{\ell}_*) \times \chEq(C^u_*,D^u_*),\]
where the projections $\chEq(C_*,D_*) \rightarrow \chEq(C^{\ell}_*,D^{\ell}_*)$ and $\chEq(C_*,D_*) \rightarrow \chEq(C^u_*,D^u_*)$ are defined by restriction.
\end{lemma}

\begin{proof}
Since $C_*$ and $D_*$ split, restrictions define an isomorphism 
\[\Hom(C_*,D_*) \cong \Hom(C^{\ell}_*,D^{\ell}_*) \oplus \Hom(C^u_*,D^u_*).\]

Let $f,f'\colon C_* \rightarrow D_*$ be chain maps. The restrictions of a chain homotopy between $f$ and $f'$ are homotopies between $f \big| _{C^{\ell}_*}$ and $f' \big| _{C^{\ell}_*}$ and between $f \big| _{C^u_*}$ and $f' \big| _{C^u_*}$, respectively, because $d^D_{\left\lceil n/2 \right\rceil}=0$ and $d^C_{\left\lfloor n/2 \right\rfloor+1}=0$. Conversely, a pair of homotopies between $f \big| _{C^{\ell}_*}$ and $f' \big| _{C^{\ell}_*}$ and between $f \big| _{C^u_*}$ and $f' \big| _{C^u_*}$ can be combined to a homotopy between $f$ and $f'$. Therefore there is an induced isomorphism $(\Hom(C_*,D_*) / {\simeq}) \cong (\Hom(C^{\ell}_*,D^{\ell}_*) / {\simeq}) \oplus (\Hom(C^u_*,D^u_*) / {\simeq})$ on homotopy classes. 

This isomorphism is defined for every pair $(C_*,D_*)$ of split chain complexes, it is compatible with the composition maps (induced by $\Hom(D_*,E_*) \times \Hom(C_*,D_*) \to \Hom(C_*,E_*)$, and similarly for $\Hom(C^{\ell}_*,D^{\ell}_*)$ and $\Hom(C^u_*,D^u_*)$) and $[\id_{C_*}] \in (\Hom(C_*,C_*) / {\simeq})$ corresponds to $([\id_{C^{\ell}_*}],[\id_{C^u_*}])$. Since $\chEq(C_*,D_*)$ consists of those elements of $(\Hom(C_*,D_*) / {\simeq})$ which have an inverse in $(\Hom(D_*,C_*) / {\simeq})$ (and similarly for $\chEq(C^{\ell}_*,D^{\ell}_*)$ and $\chEq(C^u_*,D^u_*)$), this implies that the isomorphism above restricts to a bijection $\chEq(C_*,D_*) \cong \chEq(C^{\ell}_*,D^{\ell}_*) \times \chEq(C^u_*,D^u_*)$.
\end{proof}

Recall that the orientation character $w$ determines an involution on $\Z G$, which allows us to define the dual of a left $\Z G$-module as another left $\Z G$-module. Note that a chain complex $C_*$ splits if and only if the dual cochain complex $C^*$ (defined by $C^i = (C_i)^*$) splits. 

\begin{definition} \label{def:ch-deg1}
Let $C_*$ and $D_*$ be chain complexes, $C^*$ and $D^*$ the dual cochain complexes, and let $P\colon C^{n-*} \rightarrow C_*$ and $Q\colon D^{n-*} \rightarrow D_*$ be chain homotopy equivalences. Then define $\chEq(C_*,D_*)_{P,Q} \subseteq \chEq(C_*,D_*)$ to be the subset consisting of those chain homotopy equivalences $f\colon C_*\rightarrow D_*$ which make the diagram
\[
\xymatrix{
C^{n-*} \ar[d]_{P} & D^{n-*} \ar[l]_-{f^*} \ar[d]^{Q} \\
C_* \ar[r]^-{f} & D_*
}
\]
of chain complexes commute up to chain homotopy, where $f^*$ denotes the dual of $f$.
\end{definition}

\begin{lemma} \label{lem:cheq-diag}
Let $C_*$ and $D_*$ be chain complexes, $C^*$ and $D^*$ the dual cochain complexes, and let $P\colon C^{n-*} \rightarrow C_*$ and $Q\colon D^{n-*} \rightarrow D_*$ be chain homotopy equivalences. Suppose that $C_*$ and $D_*$ split, so that $P$ and $Q$ restrict to chain homotopy equivalences $P \big| _{(C^{\ell})^{n-*}}\colon (C^{\ell})^{n-*} \rightarrow C^u_*$ and $Q \big| _{(D^{\ell})^{n-*}}\colon (D^{\ell})^{n-*} \rightarrow D^u_*$, and let $\alpha = \tau(P \big| _{(C^{\ell})^{n-*}})$ and $\beta = \tau(Q \big| _{(D^{\ell})^{n-*}}) \in \Wh(G,w)$. Then there is a commutative diagram
\[
\hspace{15mm}
\xymatrix{
\chEq(C_*,D_*)_{P,Q} \ar[r]^-{\tau} \ar[d] & \Wh(G,w) \\
\chEq(C^{\ell}_*,D^{\ell}_*) \ar[r]^-{\tau} & \Wh(G,w) \ar[u]_{x \mapsto x-(-1)^n\overline{x} + \beta - \alpha}
}
\]
where the vertical map on the left is given by restriction, i.e.\ it is the composition of the inclusion $\chEq(C_*,D_*)_{P,Q} \rightarrow \chEq(C_*,D_*)$ and the projection $\chEq(C_*,D_*) \rightarrow \chEq(C^{\ell}_*,D^{\ell}_*)$ from \cref{lem:split-cheq}.
\end{lemma}

\begin{proof}
Let $f \in \chEq(C_*,D_*)_{P,Q}$. Since $C_*$ and $D_*$ split, by \cref{lem:WT-ch-add}  we have that $\tau(f) = \tau(f \big| _{C^{\ell}_*}) + \tau(f \big| _{C^u_*})$, so it is enough to prove that $\tau(f \big| _{C^u_*}) = -(-1)^n\overline{\tau(f \big| _{C^{\ell}_*})} + \beta - \alpha$. 

The diagram in \cref{def:ch-deg1} restricts to a homotopy commutative diagram
\[
\xymatrix{
(C^{\ell})^{n-*} \ar[d]_{P |_{(C^{\ell})^{n-*}}} & & (D^{\ell})^{n-*} \ar[ll]_-{f^*|_{(D^{\ell})^{n-*}}} \ar[d]^{Q |_{(D^{\ell})^{n-*}}} \\
C^u_* \ar[rr]^-{f|_{C^u_*}} & & D^u_*
}
\] 
By \cref{prop:chain-hom-chain-equivs-same-torsion} and \cref{lem:WT-ch-comp} we have $$\tau(f^* \big| _{(D^{\ell})^{n-*}}) + \alpha + \tau(f \big| _{C^u_*}) = \beta,$$ or equivalently $$\tau(f \big| _{C^u_*}) = -\tau(f^* \big| _{(D^{\ell})^{n-*}}) + \beta - \alpha.$$ Therefore it suffices to prove that $$\tau(f^* \big| _{(D^{\ell})^{n-*}}) = (-1)^n\overline{\tau(f \big| _{C^{\ell}_*})}.$$

By \cref{lem:WT-ch-shift} we have $\tau(f^* \big| _{(D^{\ell})^{n-*}}) = (-1)^n \tau(f^* \big| _{(D^{\ell})^{-*}})$, and it follows from \cref{def:WT-ch-cochain} that $\tau(f^* \big| _{(D^{\ell})^{-*}}) = \tau(f^* \big| _{(D^{\ell})^*})$. Finally by \cref{lem:WT-ch-dual}, we have $\tau(f^* \big| _{(D^{\ell})^*}) = \overline{\tau(f \big| _{C^{\ell}_*})}$, which completes the proof. 
\end{proof}

\section{Thickenings} \label{s:thick}

We recall Wall's definition of thickenings and some of their basic properties. Fix positive integers $n, k$ with $n \geq 2k+1$, and let $K$ be a connected finite CW complex of dimension (at most) $k$. 

\begin{definition}
Suppose that $n \geq \max(k+3,2k+1)$. An $n$-dimensional \emph{thickening} of $K$ is a simple homotopy equivalence $f_T\colon K \rightarrow T$, where $T$ is an $n$-manifold with boundary such that the inclusion map $\partial T \hookrightarrow T$ induces an isomorphism $\pi_1(\partial T) \cong \pi_1(T)$.

Two thickenings $f_T\colon K \rightarrow T$ and $f_{T'}\colon K \rightarrow T'$ are equivalent if there is a diffeomorphism $H\colon T \rightarrow T'$ such that $H \circ f_T \simeq f_{T'}$.
\end{definition}

Since $n \geq 2k+1$, we can assume that $f_T$ is an embedding.
From now on thickening will mean $n$-dimensional thickening, unless indicated otherwise.

\begin{lemma}[Wall {\cite[Proposition 5.1]{Wa66}}] \label{lem:thick-class}
Suppose that $n \geq \max(6,2k+1)$. The assignment $(f_T\colon K \rightarrow T) \mapsto f_T^*(\nu_T)$, where $\nu_T$ denotes the stable normal bundle of $T$, is a bijection between the set of equivalence classes of thickenings of $K$ and the set of isomorphism classes of stable vector bundles over $K$. 
\qed
\end{lemma}

\begin{remark} \label{rem:thick-5}
The arguments of \cite[{\S}5, {\S}7 and {\S}8]{Wa66} also show that if $n = 5 \geq 2k+1$, then the assignment $(f_T \colon K \rightarrow T) \mapsto f_T^*(\nu_T)$ is surjective. 
\end{remark}

\begin{lemma}[Wall] \label{lem:thick-emb}
Suppose that $n \geq \max(6,2k+1)$, $M$ is an $n$-manifold (possibly with boundary), and $\varphi \colon K \rightarrow M$ is a continuous map. Then there is a thickening $f_T \colon K \rightarrow T$ and an embedding $i \colon T \rightarrow \interior M$ such that $i \circ f_T \simeq \varphi$. The thickening $f_T$ is unique up to equivalence. 
\end{lemma}

\begin{proof}
The existence part follows from Wall's embedding theorem \cite{Wa66}. For the uniqueness assume that there is another thickening $f_{T'} \colon K \rightarrow T'$ and an embedding $i' \colon T' \rightarrow M$ such that $i' \circ f_{T'} \simeq \varphi$. Then $f_T^*(\nu_T) \cong f_T^*(i^*(\nu_M)) \cong \varphi^*(\nu_M) \cong f_{T'}^*((i')^*(\nu_M)) \cong f_{T'}^*(\nu_{T'})$, so by \cref{lem:thick-class} it follows that $f_T$ and $f_{T'}$ are equivalent. 
\end{proof}

\begin{lemma} \label{lem:thick-conn}
Suppose that $n \geq \max(k+3,2k+1)$ and $f_T \colon K \rightarrow T$ is a thickening. Then the inclusion $\partial T \rightarrow T$ is $(n-k-1)$-connected.
\end{lemma}

\begin{proof}
By the definition of thickenings $\pi_1(\partial T) \rightarrow \pi_1(T)$ is an isomorphism. Let $G = \pi_1(T)$ and consider homology with $\Z G$ coefficients. We have $H_i(T, \partial T) \cong H^{n-i}(T) \cong H^{n-i}(K) = 0$ if $n-i>k$, equivalently, $i \leq n-k-1$. 
\end{proof}

\begin{lemma} \label{lem:fdt}
Suppose that $n \geq \max(6,2k+2)$ and $f_T \colon K \rightarrow T$ is an $n$-dimensional thickening. Then there is a map $f_{\partial T} \colon K \rightarrow \partial T$ such that $f_{\partial T} \simeq f_T$ as maps $K \rightarrow T$. Furthermore, $f_{\partial T}$ is unique up to homotopy.
\end{lemma}

\begin{proof}
By \cite[{\S}5]{Wa66} there is an $(n-1)$-dimensional thickening $f_V \colon K \rightarrow V$ and a diffeomorphism $T \approx V \times I$ such that $f_T$ is homotopic to the composition $K \stackrel{f_V}{\longrightarrow} V \rightarrow V \times I \approx T$. Hence $\partial T \approx V \cup_{\id_{\partial V}} V$, and we can take $f_{\partial T}$ to be the composition $K \stackrel{f_V}{\longrightarrow} V \rightarrow V \cup_{\id_{\partial V}} V \approx \partial T$. Since the inclusion $\partial T \rightarrow T$ is at least $(k+1)$-connected, any two maps from $K$ to $\partial T$ that are homotopic as $K \rightarrow T$ maps are also homotopic as $K \rightarrow \partial T$ maps. Therefore $f_{\partial T}$ is unique up to homotopy. 
\end{proof}

\begin{lemma} \label{lem:compl-he} 
Suppose that $n \geq \max(6,2k+2)$. Let $M$ be a closed $n$-manifold, $f_T \colon K \rightarrow T$ a thickening, $\varphi \colon K \rightarrow M$ an $\left\lfloor \frac{n}{2} \right\rfloor$-connected map and $i \colon T \rightarrow M$ an embedding such that $i \circ f_T \simeq \varphi$. Let $C = M \setminus i(\interior T)$ denote the complement of $i(T)$, let $j \colon C \rightarrow M$ be the inclusion, and let $\varphi'$ denote the composition $K \stackrel{f_{\partial T}}{\longrightarrow} \partial T \stackrel{i}{\longrightarrow} i(\partial T) \rightarrow C$. 
\benum
\item\label{item:lem-compl-he-a} The map $\varphi' \colon K \rightarrow C$ is a homotopy equivalence.
\item\label{item:lem-compl-he-b} A map $f \colon K \rightarrow C$ is homotopic to $\varphi'$ if and only if $j \circ f \simeq \varphi \colon K \rightarrow M$.
\eenum
\end{lemma}

\begin{proof}
\eqref{item:lem-compl-he-a} Since $\varphi$ is $\left\lfloor \frac{n}{2} \right\rfloor$-connected and $f_T$ is a homotopy equivalence, $i$ is $\left\lfloor \frac{n}{2} \right\rfloor$-connected too. The inclusions $\partial T \rightarrow T$ and $i$ induce isomorphisms on $\pi_1$, so it follows from the Van Kampen theorem that the inclusions $i(\partial T) = \partial C \rightarrow C$ and $j$ also induce isomorphisms on $\pi_1$. Let $G = \pi_1(M)$, identify the fundamental groups of $T$, $\partial T$ and $C$ with $G$ via the inclusions, and consider their homology with $\Z G$ coefficients. 

By excision and \cref{lem:thick-conn} we have $H_r(M,C) \cong H_r(T, \partial T) = 0$ if $r \leq n-k-1$. Therefore $H_r(j) \colon H_r(C) \rightarrow H_r(M)$ is an isomorphism if $r \leq n-k-2$, in particular if $r \leq \left\lceil \frac{n}{2} \right\rceil-1$. The induced homomorphism $H_r(\varphi) \colon H_r(K) \rightarrow H_r(M)$ is an isomorphism for $r \leq \left\lfloor \frac{n}{2} \right\rfloor$, because $\varphi$ is $\left\lfloor \frac{n}{2} \right\rfloor$-connected and $K$ has dimension $k < \left\lfloor \frac{n}{2} \right\rfloor$. Since $\varphi \simeq i \circ f_T \simeq i \circ f_{\partial T} \simeq j \circ \varphi' \colon K \rightarrow M$, we get that $H_r(\varphi') \colon H_r(K) \rightarrow H_r(C)$ is an isomorphism if $r \leq \left\lceil \frac{n}{2} \right\rceil-1$. We also have that $H_r(C) \cong H^{n-r}(C, \partial C) \cong H^{n-r}(M, i(T)) = 0$ and $H_r(K)=0$ if $r \geq \left\lceil \frac{n}{2} \right\rceil$ (hence $n-r \leq \left\lfloor \frac{n}{2} \right\rfloor$). Therefore $\varphi'$ induces an isomorphism on $\pi_1$ and all homology groups, so it is a homotopy equivalence. 

\eqref{item:lem-compl-he-b} We already saw that $\varphi \simeq j \circ \varphi'$. This implies the only if direction. Also,  by part \eqref{item:lem-compl-he-a} this implies that $j$ is $\left\lfloor \frac{n}{2} \right\rfloor$-connected. Hence if $j \circ f \simeq \varphi \simeq j \circ \varphi'$ for some $f \colon K \rightarrow C$, then $f \simeq \varphi'$, because $K$ has dimension $k \leq \left\lfloor \frac{n}{2} \right\rfloor - 1$.
\end{proof}

The next lemma is implicit in the discussion on \cite[p.\ 77]{Wa66}.

\begin{lemma} \label{lem:glue-hcob}
Let $M$ be an $n$-manifold with boundary, and let $N \subset \interior M$ be a codimension $0$ submanifold. Then the following hold. 
\benum
\item\label{item:lem-glue-hcob-a} If $M \setminus \interior N$ is an $h$-cobordism between $\partial N$ and $\partial M$, then the inclusion $i \colon N \rightarrow M$ is a homotopy equivalence. 

\item\label{item:lem-glue-hcob-b} If the inclusions induce isomorphisms $\pi_1(\partial N) \cong \pi_1(N)$ and $\pi_1(\partial M) \cong \pi_1(M)$, and $i \colon N \rightarrow M$ is a homotopy equivalence, then $M \setminus \interior N$ is an $h$-cobordism. 

\item\label{item:lem-glue-hcob-c} If the assumptions in \eqref{item:lem-glue-hcob-b} hold, then $\tau(i) = j_*(\tau(M \setminus \interior N, \partial N)) \in \Wh(\pi_1(M))$, where $j \colon M \setminus \interior N \rightarrow M$ is the inclusion. 
\eenum
\end{lemma}

\begin{proof}
\eqref{item:lem-glue-hcob-a} 
Since the inclusion $\partial N \rightarrow M \setminus \interior N$ is a homotopy equivalence, by applying homotopy excision \cite[Theorem 4.23]{Ha02} to the map $(M \setminus \interior N, \partial N) \to (M,N)$ we deduce that $i$ is a homotopy equivalence too.

\eqref{item:lem-glue-hcob-b} Let $G = \pi_1(M)$, it follows from the Van Kampen theorem and the assumptions that the inclusions identify $\pi_1(N)$, $\pi_1(\partial N)$, $\pi_1(M \setminus \interior N)$ and $\pi_1(\partial M)$ with $G$, in particular $\pi_1(\partial N) \cong \pi_1(M \setminus \interior N) \cong \pi_1(\partial M)$. Consider homology with $\Z G$ coefficients. Since $i$ is a homotopy equivalence, we have $H_*(M \setminus \interior N, \partial N) \cong H_*(M,N) = 0$ by excision. By Poincar\'e duality we also have $H_*(M \setminus \interior N, \partial M) = 0$, therefore the inclusions $\partial N \rightarrow M \setminus \interior N$ and $\partial M \rightarrow M \setminus \interior N$ are both homotopy equivalences.

\eqref{item:lem-glue-hcob-c} Fix a CW structure on $\partial N$ and extend it to a CW structure on $M$ (so that $\partial N$ is a subcomplex of $N$ and $M \setminus \interior N$). Consider cellular chain complexes with $\Z G$ coefficients. We have a commutative diagram 
\[
\xymatrix{
0 \ar[r] & C_*(\partial N) \ar[r] \ar[d] & C_*(N) \ar[r] \ar[d]_-{i_*} & C_*(N, \partial N) \ar[r] \ar[d] & 0 \\
0 \ar[r] & C_*(M \setminus \interior N) \ar[r] & C_*(M) \ar[r] & C_*(M,M \setminus \interior N) \ar[r] & 0
}
\]
where the rows are short exact sequences of chain complexes. The left and middle vertical maps are chain homotopy equivalences by our assumptions and part \eqref{item:lem-glue-hcob-b}. The vertical map on the right is a chain homotopy equivalences with vanishing Whitehead torsion, because the cells in $N \setminus \partial N$ determine the standard basis in both $C_*(N, \partial N)$ and $C_*(M,M \setminus \interior N)$. So by \cref{lem:WT-ch-add} we have $\tau(i_*) = \tau(C_*(\partial N) \rightarrow C_*(M \setminus \interior N)) \in \Wh(G)$. 
\end{proof}

\section{Polarised doubles} \label{s:gd}

Fix positive integers $n, k$ with $n \geq 2k+1$, and let $M$ be a closed $n$-manifold, $K$ a finite CW complex of dimension (at most) $k$ and let $\varphi \colon K \rightarrow M$ be a continuous map. We will think of (the homotopy class of) $\varphi$ as extra structure on $M$ and call the pair $(M,\varphi)$ a polarised manifold.

\subsection{Double structures on polarised manifolds}

We start by defining the different types of double structures that a polarised manifold may have.

\begin{definition} \label{def:doubles}
Let $(M,\varphi)$ be a polarised manifold. 
\begin{compactitem}[$\bullet$]
\item A \emph{generalised double structure} on $(M,\varphi)$ is a diffeomorphism $h \colon T \cup_{g_0} W \cup_{g_1} T \rightarrow M$ such that $\varphi \simeq h \circ i_1 \circ f_T$, where $f_T \colon K \rightarrow T$ is a thickening, $W$ is an $h$-cobordism with $\partial W = \partial_0 W \sqcup \partial_1 W$, $g_0 \colon \partial_0 W \rightarrow \partial T$ and $g_1 \colon \partial T \rightarrow \partial_1 W$ are diffeomorphisms and $i_1 \colon T \rightarrow T \cup_{g_0} W \cup_{g_1} T$ is the inclusion of the first component.
\item A \emph{twisted double structure} on $(M,\varphi)$ is a diffeomorphism $h \colon T \cup_g T \rightarrow M$ such that $\varphi \simeq h \circ i_1 \circ f_T$, where $f_T \colon K \rightarrow T$ is a thickening, $g \colon \partial T \rightarrow \partial T$ is a diffeomorphism, and $i_1 \colon T \rightarrow T \cup_g T$ is the inclusion of the first component.
\item A \emph{trivial double structure} on $(M,\varphi)$ is a twisted double structure with $g = \id_{\partial T}$.
\end{compactitem}
\end{definition}

Of course, $(M,\varphi)$ has a twisted double structure if and only if it has a generalised double structure with $W \approx \partial T \times I$. If the manifold $M$ is of the form $T \cup_{g_0} W \cup_{g_1} T$, $T \cup_g T$, or $T \cup_{\id_{\partial T}} T$, then $\id_M$ is a generalised/twisted/trivial double structure on $(M,i_1 \circ f_T)$. 

\begin{proposition} \label{prop:double-conn}
Suppose that $n \geq \max(6,2k+2)$. Then $(M,\varphi)$ has a generalised double structure if and only if $\varphi$ is $\left\lfloor \frac{n}{2} \right\rfloor$-connected.
\end{proposition}

\begin{proof}
First assume that a generalised double structure $h$ exists as in \cref{def:doubles}. By \cref{lem:glue-hcob}~\eqref{item:lem-glue-hcob-a}, the inclusion $T \rightarrow T \cup W$ is a homotopy equivalence. By \cref{lem:thick-conn} the pair $(T, \partial T)$ is $(n-k-1)$-connected, and by homotopy excision 
the same holds for the pair $(T \cup W \cup T, T \cup W)$. Therefore $i_1$ is $(n-k-1)$-connected too. Since $\varphi \simeq h \circ i_1 \circ f_T$ and $h$ and $f_T$ are homotopy equivalences, $\varphi$ is also $(n-k-1)$-connected, in particular it is $\left\lfloor \frac{n}{2} \right\rfloor$-connected. 

Now assume that $\varphi$ is $\left\lfloor \frac{n}{2} \right\rfloor$-connected. By \cref{lem:thick-emb} there is a thickening $f_T \colon K \rightarrow T$ of $K$ and an embedding $i \colon T \rightarrow M$ such that $i \circ f_T \simeq \varphi$. Let $C = M \setminus i(\interior T)$ denote the complement of $i(T)$. By \cref{lem:compl-he} the composition $K \stackrel{f_{\partial T}}{\longrightarrow} \partial T \stackrel{i}{\longrightarrow} i(\partial T) \rightarrow C$, denoted by $\varphi'$, is a homotopy equivalence. Applying \cref{lem:thick-emb} again, there is a thickening $f_{T'} \colon K \rightarrow T'$ of $K$ and an embedding $i' \colon T' \rightarrow \interior C$ such that $i' \circ f_{T'} \simeq \varphi'$. Moreover, since $\varphi$ and $\varphi'$ are homotopic as $K \rightarrow M$ maps, $f_T$ and $f_{T'}$ are equivalent thickenings, i.e.\ there is a diffeomorphism $H \colon T \rightarrow T'$ such that $H \circ f_T \simeq f_{T'}$. Let $W = C \setminus i'(\interior T')$. The embedding $i' \colon T' \rightarrow C$ is a homotopy equivalence, because $\varphi' \colon K \rightarrow C$ and $f_{T'}$ are both homotopy equivalences and $i' \circ f_{T'} \simeq \varphi'$. So by \cref{lem:glue-hcob}~\eqref{item:lem-glue-hcob-b} $W$ is an $h$-cobordism. Therefore the decomposition $M = i(T) \cup W \cup (i' \circ H)(T)$ determines a generalised double structure on $(M,\varphi)$. 
\end{proof}

\begin{corollary} \label{cor:gds-he}
Suppose that $n \geq \max(6,2k+2)$, and let $N$ be an $n$-manifold. If $(M,\varphi)$ has a generalised double structure and $M \simeq N$, then $(N,\psi)$ also has a generalised double structure for some $\psi \colon K \rightarrow N$.
\end{corollary}

\begin{proof}
If $h \colon M \rightarrow N$ is a homotopy equivalence, then $\psi  \coloneqq  h \circ \varphi$ is $\left\lfloor \frac{n}{2} \right\rfloor$-connected.
\end{proof}

\subsection{The Whitehead torsion of a polarised double}

Next we define and study the invariant $\tau(M,\varphi)$, the torsion associated to  a polarised double. 

\begin{definition} \label{def:tau-pol}
Suppose that $n \geq \max(6,2k+2)$ and $\varphi$ is $\left\lfloor \frac{n}{2} \right\rfloor$-connected. The \emph{Whitehead torsion associated to $(M,\varphi)$}, denoted by $\tau(M,\varphi)$, is defined as follows. Let $f_T \colon K \rightarrow T$ be a thickening and let $i \colon T \rightarrow M$ be an embedding such that $i \circ f_T \simeq \varphi$. Let $C = M \setminus i(\interior T)$ denote the complement of $i(T)$ with inclusion $j \colon C \rightarrow M$, and let $\varphi'$ be the composition 
\[\varphi' \colon K \stackrel{f_{\partial T}}{\longrightarrow} \partial T \stackrel{i}{\longrightarrow} i(\partial T) \rightarrow C.\] Then 
\[
\tau(M,\varphi) = j_*(\tau(\varphi')) \in \Wh(\pi_1(M)) \text{.}
\] 
\end{definition}

\begin{proposition}
The Whitehead torsion  $\tau(M,\varphi)$ is well-defined. 
\end{proposition}

\begin{proof}
First we check that the assumptions made in \cref{def:tau-pol} are satisfied. A suitable thickening $f_T$ and embedding $i$ exist by \cref{lem:thick-emb}. The composition $\varphi' \colon K \rightarrow C$ is a homotopy equivalence by \cref{lem:compl-he}, so $\tau(\varphi') \in \Wh(\pi_1(C))$ is defined. Since $j \circ \varphi' \simeq \varphi$ and $\pi_1(\varphi)$ and $\pi_1(\varphi')$ are isomorphisms, $j_* \colon \Wh(\pi_1(C)) \rightarrow \Wh(\pi_1(M))$ is an isomorphism too. 

Next we check that $\tau(M,\varphi)$ is independent of choices. By \cref{lem:thick-emb}, the thickening $f_T$ is well-defined (up to equivalence). Suppose that $i_0 \colon T \rightarrow M$ and $i_1 \colon T \rightarrow M$ are two embeddings with $i_0 \circ f_T \simeq i_1 \circ f_T \simeq \varphi$ and corresponding $C_0,j_0,\varphi'_0$ and $C_1,j_1,\varphi'_1$. We will prove that 
\[
(j_0)_*(\tau(\varphi'_0)) = (j_1)_*(\tau(\varphi'_1)).
\] 

Since $n \geq 2k+2$, we can assume that $f_T$ is an embedding, and since the embeddings $i_0 \circ f_T$ and $i_1 \circ f_T \colon K \rightarrow M$ are homotopic, they are isotopic~\cite[Theorem~6]{Whitney-36}. 

By the isotopy extension theorem there is a diffeomorphism $H \colon M\rightarrow M$, isotopic to $\id_M$, such that $i_0 \circ f_T = H \circ i_1 \circ f_T$. Let $i_2 = H \circ i_1$, $C_2 = H(C_1) = M \setminus i_2(\interior T)$, $j_2 = H \circ j_1 \circ H^{-1} \colon C_2 \rightarrow M$ and $\varphi'_2 = H \circ \varphi'_1 \colon K \rightarrow C_2$. 

Since $H$ is a diffeomorphism, $\tau(\varphi'_2) = H_*(\tau(\varphi'_1)) \in \Wh(\pi_1(C_2))$ by \cref{prop:WT-composition} and \cref{thm:chapman}, and hence $(j_2)_*(\tau(\varphi'_2)) = H_* \circ (j_1)_*(\tau(\varphi'_1))$. We have $H_* = \id \colon \Wh(\pi_1(M)) \rightarrow \Wh(\pi_1(M))$, because $H \simeq \id_M$, therefore $(j_2)_*(\tau(\varphi'_2)) = (j_1)_*(\tau(\varphi'_1))$. So it is enough to prove that 
\[
(j_0)_*(\tau(\varphi'_0)) = (j_2)_*(\tau(\varphi'_2)).
\]

The image of $i_0 \circ f_T = i_2 \circ f_T \colon K \to M$ is contained in $i_0(T) \cap i_2(T)$, so we can apply \cref{lem:thick-emb} to get an embedding $i_3 \colon T \rightarrow i_0(T) \cap i_2(T)$ such that $i_3 \circ f_T \simeq i_0 \circ f_T$, and define the corresponding $C_3 := M \setminus i_3(\interior T)$, with inclusion $j_3 \colon C_3 \to M$, and $\varphi'_3 \colon K \to C_3$.  Since $i_3(T) \subseteq i_0(T)$, we have $C_3 \supseteq C_0$.  Let $h \colon i_3(T) \rightarrow i_0(T)$ and $h' \colon C_0 \rightarrow C_3$ denote the inclusions (hence $j_0 = j_3 \circ h'$).

By construction, $h \circ i_3 \circ f_T \simeq i_0 \circ f_T \colon K \rightarrow i_0(T)$. Since $f_T$ and the diffeomorphisms $i_0 \colon T \rightarrow i_0(T)$ and $i_3 \colon T \rightarrow i_3(T)$ are simple homotopy equivalences, $h$ is a simple homotopy equivalence too. By \cref{lem:glue-hcob}~\eqref{item:lem-glue-hcob-b} and~\eqref{item:lem-glue-hcob-c}, we see that~$i_0(T) \setminus i_3(\interior T)$ is an $s$-cobordism, and this in turn implies that $h'$ is a simple homotopy equivalence. We have $j_3 \circ h' \circ \varphi'_0 = j_0 \circ \varphi'_0 \simeq \varphi$, so by \cref{lem:compl-he}~\eqref{item:lem-compl-he-b}, we have that $h' \circ \varphi'_0 \simeq \varphi'_3 \colon K \rightarrow C_3$. Use this, \cref{prop:WT-composition}, and the fact that $\tau(h')=0$, to deduce that 
\[
(j_3)_*(\tau(\varphi'_3)) = (j_3)_*(\tau(h' \circ \varphi'_0)) = (j_3)_*(h'_*(\tau(\varphi'_0))) = (j_0)_*(\tau(\varphi'_0)).
\]

Since also $i_3(T) \subseteq i_2(T)$, we can prove similarly that $(j_3)_*(\tau(\varphi'_3)) = (j_2)_*(\tau(\varphi'_2))$. Therefore $(j_0)_*(\tau(\varphi'_0)) = (j_2)_*(\tau(\varphi'_2))$ as required.
\end{proof}

Note that $\tau(M,\varphi)$ is defined precisely when $(M,\varphi)$ admits a (generalised) double structure, but it is determined by the polarised manifold $(M,\varphi)$ alone, and is independent of a choice of a double structure. However, when a double structure is present, it can be used to compute $\tau(M,\varphi)$, as we explain below. More generally, we will investigate the relationship between $\tau(M,\varphi)$ and the possible double structures on $(M,\varphi)$. 

\begin{proposition} \label{prop:tau-double-comp}
Suppose that $n \geq \max(6,2k+2)$ and $(M,\varphi)$ has a generalised double structure $h \colon T \cup_{g_0} W \cup_{g_1} T \rightarrow M$, and let $i_1$, $i_2$, and $i_3$ denote the inclusions of the three components of $T \cup_{g_0} W \cup_{g_1} T$. 
\benum
\item\label{item:prop:tau-double-comp-a} There is a homotopy automorphism $\alpha \colon K \rightarrow K$ such that $\varphi \circ \alpha \simeq h \circ i_3 \circ f_T \colon K \rightarrow M$, and such an $\alpha$ is unique up to homotopy.

\item\label{item:prop:tau-double-comp-b} This $\alpha$ satisfies $\alpha^*(\varphi^*(\nu_M)) \cong \varphi^*(\nu_M)$.

\item\label{item:prop:tau-double-comp-c} For this $\alpha$, we have $\tau(M,\varphi) = (h \circ i_2)_*(\tau(W,\partial_1 W)) - \varphi_*(\tau(\alpha))$.
\eenum
\end{proposition}

\begin{proof}
First we introduce some notation. Let $i = h \circ i_1 \colon T \rightarrow M$, $C = h(W \cup_{g_1} T)$, and $\varphi' = h \circ i_1 \circ f_{\partial T} \colon K \rightarrow C$, and let $j \colon C \rightarrow M$ denote the inclusion. Then, by \cref{def:tau-pol}, we have $\tau(M,\varphi) = j_*(\tau(\varphi')) = h_*(\tau(\varphi''))$, where $\varphi'' = i_1 \circ f_{\partial T} \colon K \rightarrow W \cup_{g_1} T$. 

\eqref{item:prop:tau-double-comp-a} By \cref{lem:compl-he}, $\varphi'$ is a homotopy equivalence, and since $h \big| _{W \cup_{g_1} T} \colon W \cup_{g_1} T \rightarrow C$ is a diffeomorphism, $\varphi''$ is a homotopy equivalence too. Let \[\alpha := (\varphi'')^{-1} \circ i_3 \circ f_T \colon K \rightarrow K,\] where $i_3$ is regarded as an inclusion $T \rightarrow W \cup_{g_1} T$ and $(\varphi'')^{-1} \colon W \cup_{g_1} T \rightarrow K$ is the homotopy inverse of $\varphi''$. By \cref{lem:glue-hcob}~\eqref{item:lem-glue-hcob-a}, $i_3$ is a homotopy equivalence, so $\alpha$ is a homotopy equivalence too. Moreover, we compute that 
\[\varphi \circ \alpha \simeq h \circ i_1 \circ f_T \circ \alpha \simeq h \circ i_1 \circ f_{\partial T} \circ \alpha = h \circ \varphi'' \circ \alpha \simeq h \circ i_3 \circ f_T \colon K \rightarrow M.\]

If $\alpha, \alpha' \colon K \rightarrow K$ are two maps such that $\varphi \circ \alpha \simeq \varphi \circ \alpha' \colon K \rightarrow M$, then $\alpha \simeq \alpha'$, because $\varphi$ is $\left\lfloor \frac{n}{2} \right\rfloor$-connected and $K$ has dimension $k \leq \left\lfloor \frac{n}{2} \right\rfloor - 1$. 

\eqref{item:prop:tau-double-comp-b} Since $\varphi \circ \alpha \simeq h \circ i_3 \circ f_T$, we have $\alpha^*(\varphi^*(\nu_M)) \cong f_T^*((h \circ i_3)^*(\nu_M)) \cong f_T^*(\nu_T)$, where we used that $h \circ i_3 \colon T \rightarrow M$ is a codimension $0$ embedding. Since $\varphi \simeq h \circ i_1 \circ f_T$, and $h \circ i_1 \colon T \rightarrow M$ is a codimension $0$ embedding, we similarly obtain $\varphi^*(\nu_M) \cong f_T^*(\nu_T)$. Hence $\alpha^*(\varphi^*(\nu_M)) \cong \varphi^*(\nu_M)$.

\eqref{item:prop:tau-double-comp-c} It follows from the definition of $\alpha$ that $\varphi'' \simeq i_3 \circ f_T \circ \alpha^{-1} \colon K \rightarrow W \cup_{g_1} T$.
We have \[\tau(i_3 \circ f_T \circ \alpha^{-1}) = \tau(i_3) + (i_3)_*(\tau(f_T)) + (i_3 \circ f_T)_*(\tau(\alpha^{-1}))\] by \cref{prop:WT-composition}. The map $f_T$ is a simple homotopy equivalence, and by \cref{lem:glue-hcob}~\eqref{item:lem-glue-hcob-c}, we have that  $\tau(i_3) = (i_2)_*(\tau(W,\partial_1 W))$, where $i_2$ is regarded as an embedding $W \rightarrow W \cup_{g_1} T$. We also have $0 = \tau(\alpha^{-1} \circ \alpha) = \tau(\alpha^{-1}) + \alpha^{-1}_*(\tau(\alpha))$, so 
\[\tau(i_3 \circ f_T \circ \alpha^{-1}) = (i_2)_*(\tau(W,\partial_1 W)) - (i_3 \circ f_T \circ \alpha^{-1})_*(\tau(\alpha)).\] 
Since $\varphi \simeq h \circ i_3 \circ f_T \circ \alpha^{-1}$, we obtain \[\tau(M,\varphi) = h_*(\tau(\varphi'')) = h_*(\tau(i_3 \circ f_T \circ \alpha^{-1})) = (h \circ i_2)_*(\tau(W,\partial_1 W)) - \varphi_*(\tau(\alpha)).\qedhere\] 
\end{proof}

\begin{corollary} \label{cor:tau-tw-double}
Suppose that $n \geq \max(6,2k+2)$ and $(M,\varphi)$ has a twisted double structure $h \colon T \cup_g T \rightarrow M$. Let $\alpha \in \hAut(K)$ be the image of $g$ under the restriction map $\hAut(\partial T) \rightarrow \hAut(K)$ $($see \cref{rem:dt-sp}, \cref{lem:heq-restr} and \cref{rem:heq-restr}$)$. Then $\tau(M,\varphi) = -\varphi_*(\tau(\alpha))$.
\end{corollary}

\begin{proof}
Since the twisted double structure of $(M,\varphi)$ determines a generalised double structure with $W \approx \partial T \times I$, it is enough to show that $\varphi \circ \alpha \simeq h \circ i_2 \circ f_T \colon K \rightarrow M$, where $i_2 \colon T \rightarrow T \cup_g T$ is the inclusion of the second component. This holds, because 
\[\varphi \circ \alpha \simeq h \circ i_1 \circ f_T \circ \alpha \simeq h \circ i_1 \circ f_{\partial T} \circ \alpha \simeq h \circ i_1 \circ g \circ f_{\partial T} = h \circ i_2 \circ f_{\partial T} \simeq h \circ i_2 \circ f_T.\qedhere\] 
\end{proof}

\begin{corollary} \label{cor:tau-triv-double}
Suppose that $n \geq \max(6,2k+2)$ and $(M,\varphi)$ has a trivial double structure. Then $\tau(M,\varphi) = 0$.
\end{corollary}

\begin{proof}
We can apply \cref{cor:tau-tw-double} with $g = \Id_{\partial T}$, and hence $\alpha = \Id_K$.
\end{proof}

\begin{remark}
The converse does not hold. There are even simply-connected counterexamples, e.g.\ take $K$ to be a point and $M$ an exotic sphere.
\end{remark}

\begin{proposition}\label{prop:characterises-twisted-double-structures}
Suppose that $n \geq \max(6,2k+2)$ and $\varphi$ is $\left\lfloor \frac{n}{2} \right\rfloor$-connected. Then $(M,\varphi)$ has a twisted double structure if and only if there is a homotopy automorphism $\alpha \in \hAut(K)$ such that there is an isomorphism of stable vector bundles $\alpha^*(\varphi^*(\nu_M)) \cong \varphi^*(\nu_M)$ and $\tau(M,\varphi) = -\varphi_*(\tau(\alpha)) \in \Wh(\pi_1(M))$. 
\end{proposition}

\begin{proof}
If $(M,\varphi)$ has a twisted double structure, then it follows from \cref{prop:tau-double-comp}~\eqref{item:prop:tau-double-comp-b} and the proof of \cref{cor:tau-tw-double} that such an $\alpha$ exists.

Now suppose that there is an $\alpha$ with $\alpha^*(\varphi^*(\nu_M)) \cong \varphi^*(\nu_M)$ and $\tau(M,\varphi) = -\varphi_*(\tau(\alpha))$. Define $f_T \colon K \rightarrow T$, $i$, $C$, $j$ and $\varphi'$ as in \cref{def:tau-pol}, so that $\tau(M,\varphi) = j_*(\tau(\varphi'))$. Let $\psi := \varphi' \circ \alpha \colon K \rightarrow C$. 
Then 
\[
j_*(\tau(\psi)) = j_*(\tau(\varphi') + \varphi'_*(\tau(\alpha))) = \tau(M,\varphi) + (j \circ \varphi')_*(\tau(\alpha)) = \tau(M,\varphi) + \varphi_*(\tau(\alpha)) = 0
\] 
by \cref{prop:WT-composition}, \cref{lem:compl-he} \eqref{item:lem-compl-he-b}, and the hypothesis. As $j_*$ is an isomorphism, this means that $\psi$ is a simple homotopy equivalence, 
hence $\psi \colon K \rightarrow C$ is a thickening. Moreover, 
\[
\psi^*(\nu_C) \cong \alpha^*((\varphi')^*(j^*(\nu_M))) \cong \alpha^*(\varphi^*(\nu_M)) \cong \varphi^*(\nu_M) \cong f_T^*(i^*(\nu_M)) \cong f_T^*(\nu_T).
\]
So by \cref{lem:thick-class} the thickenings $\psi \colon K \rightarrow C$ and $f_T \colon K \rightarrow T$ are equivalent, i.e.\ there is a diffeomorphism $H \colon T \rightarrow C$ such that $H \circ f_T \simeq \psi$. Then $(i \cup H) \colon T \cup_g T \rightarrow M$ is a twisted double structure on $(M,\varphi)$, where $g$ is the composition 
\[
g \colon \partial T \stackrel{H}{\longrightarrow} \partial C = i(\partial T) \stackrel{i^{-1}}{\longrightarrow} \partial T.\qedhere 
\] 
\end{proof}

\begin{corollary} \label{cor:tau-0}
Suppose that $n \geq \max(6,2k+2)$ and $\varphi$ is $\left\lfloor \frac{n}{2} \right\rfloor$-connected. If $\tau(M,\varphi) = 0$, then $(M,\varphi)$ has a twisted double structure.
\end{corollary}

\begin{proof}
Take $\alpha = \Id_K$ and apply \cref{prop:characterises-twisted-double-structures}. 
\end{proof}

\subsection{SP manifolds}

Next we show that if we impose some mild restrictions on $(M,\varphi)$, then $\tau(M,\varphi)$ can be expressed in terms of the Poincar\'e duality chain homotopy equivalence.

\begin{definition} \label{def:SP}
The pair $(M,\varphi)$ is a \emph{split polarised manifold} (SP manifold for short) if at least one of the following conditions holds: 
\begin{clist}{(SP1)}
\item $n \geq \max(7,2k+2)$ and $(M,\varphi)$ has a generalised double structure;
\item $n \geq \max(6,2k+2)$ and $(M,\varphi)$ has a twisted double structure; or
\item $n \geq \max(k+3,2k+1)$ and $(M,\varphi)$ has a trivial double structure.
\end{clist}
\end{definition}

\begin{remark} \label{rem:dt-sp}
If $n \geq \max(6,2k+2)$ and $f_T \colon K \rightarrow T$ is an $n$-dimensional thickening, then $n-1 \geq \max(k+3,2k+1)$ and the proof of \cref{lem:fdt} shows that $(\partial T, f_{\partial T})$ satisfies (SP3).
\end{remark}

We extend the definition of $\tau(M,\varphi)$ to the case when $\max(6,2k+2) > n \geq \max(k+3,2k+1)$ and $(M,\varphi)$ has a trivial double structure by setting $\tau(M,\varphi) = 0$. Then $\tau(M,\varphi)$ is defined for every SP manifold $(M,\varphi)$, and if (SP3) holds, then $\tau(M,\varphi) = 0$.

\begin{theorem} \label{theorem:tau-pd}
Suppose that $(M,\varphi)$ is an SP manifold, and let $G = \pi_1(M)$ with orientation character $w \colon G \rightarrow \left\{ \pm 1 \right\}$. Then $M$ has a CW decomposition with the following properties. 
\benum
\item\label{item:theorem:tau-pd-a} The $\left\lfloor \frac{n}{2} \right\rfloor$-skeleton of $M$ is identified with $K$ via an embedding $K \rightarrow M$ homotopic to $\varphi$.

\item\label{item:theorem:tau-pd-b} Let $C_*(M) = C_*(M; \Z G)$ denote the cellular chain complex of $M$ with $\Z G$ coefficients. Then $C_*(M)$ splits $($see \cref{def:split}$)$.

\item\label{item:theorem:tau-pd-c} Let $\PD \colon C^{n-*}(M) \rightarrow C_*(M)$ denote the chain homotopy equivalence induced by Poincar\'e duality $($which is determined up to chain homotopy by a choice of twisted fundamental class $[M] \in H_n(M;\Z^w) \cong \Z$, where $\Z^w$ is the orientation module$)$. 
Then 
\[
\tau(\PD \big| _{C^{\ell}(M)^{n-*}} \colon C^{\ell}(M)^{n-*} \rightarrow C^u_*(M)) = \tau(M,\varphi) \in \Wh(G,w) \text{.}
\]
\eenum
\end{theorem}

\begin{proof}
\eqref{item:theorem:tau-pd-a} Let $h \colon T \cup_{g_0} W \cup_{g_1} T \rightarrow M$ or $h \colon T \cup_g T \rightarrow M$ be a generalised, twisted or trivial double structure on $(M,\varphi)$ if (SP1), (SP2) or (SP3) holds, respectively (in the last case $g = \id_{\partial T}$). Let $i_1$, $i_2$ and $i_3$ denote the embeddings of the components of $T \cup_{g_0} W \cup_{g_1} T$, or let $i_1$ and $i_2$ denote the embeddings of the components of $T \cup_g T$. 

The thickening $T$ has a handlebody decomposition such that the embedding $f_T \colon K \rightarrow T$ identifies $K$ with the CW complex formed by the cores of the handles, in particular there is a bijection between the $i$-handles and the $i$-cells of $K$ (see \cite[{\S}7]{Wa66}). This handlebody decomposition determines a Morse function $m_0 \colon T \to [0,1]$ such that index-$i$ critical points of $m_0$ correspond to $i$-handles and $m_0^{-1}(1) = \partial T$. By the normal form lemma \cite{Milnor-h-cob}, \cite[Lemma~1.24]{Luck-basic-intro-surgery}, if $n \geq 7$, then there is a Morse function $m_1 \colon W \to [0,1]$ on the $h$-cobordism $W$ such that all critical points have index $\left\lfloor \frac{n}{2} \right\rfloor + 1$ or $\left\lfloor \frac{n}{2} \right\rfloor + 2$, and $m_1^{-1}(0) = \partial_0 W$ and $m_1^{-1}(1) = \partial_1 W$. 

Now we can define a Morse function $m \colon M \to \R$ on $M$. In the case of (SP1) we take $m_0$ on $h \circ i_1(T)$, $m_1+1$ on $h \circ i_2(W)$ and $3-m_0$ on $h \circ i_3(T)$. In the case of (SP2) and (SP3) we take $m_0$ on $h \circ i_1(T)$ and $2-m_0$ on $h \circ i_2(T)$.

The Morse function $m$ determines a handlebody decomposition of $M$, and by \cite[Theorem 4.18]{Ma02} $M$ is homeomorphic to the CW complex formed by the cores of the handles (having one $i$-cell for each index-$i$ critical point of $m$). The critical points of $m$ have index at most $k \leq \left\lfloor \frac{n}{2} \right\rfloor$ in $h \circ i_1(T)$, $\left\lfloor \frac{n}{2} \right\rfloor + 1$ or $\left\lfloor \frac{n}{2} \right\rfloor + 2$ in $h \circ i_2(W)$, and at least $n-k \geq \left\lfloor \frac{n}{2} \right\rfloor + 1$ in $h \circ i_2(T)$ or $h \circ i_3(T)$. Therefore the $\left\lfloor \frac{n}{2} \right\rfloor$-skeleton of $M$ consists of the cores of the handles in $h \circ i_1(T)$.

Let $f_M = h \circ i_1 \circ f_T \colon K \rightarrow M$, then by the above $f_M$ identifies the $\left\lfloor \frac{n}{2} \right\rfloor$-skeleton of $M$ (which is also the $k$-skeleton) with $K$. Moreover, $f_M \simeq \varphi$ by \cref{def:doubles}.

\eqref{item:theorem:tau-pd-b} If $n \geq 2k+2$, then $k \leq \left\lfloor \frac{n}{2} \right\rfloor - 1$, so $M$ has no $\left\lfloor \frac{n}{2} \right\rfloor$-cells. This means that $C_{\left\lfloor n/2 \right\rfloor}(M) = 0$, hence $C_*(M)$ splits. 

If $n = 2k+1$, then $(M, \varphi)$ satisfies (SP3), so $h^{-1}$ is a diffeomorphism $M \rightarrow T \cup_{\id_{\partial T}} T$. There is a well-defined retraction $\id_T \cup \id_T \colon T \cup_{\id_{\partial T}} T \rightarrow T$. Since the embedding $f_T$ is a homotopy equivalence, $K$ is a deformation retract of $T$, and we can compose $h^{-1}$ with the two retractions to get a retraction $r \colon M \to K$. It induces a chain map $C_*(r) \colon C_*(M) \rightarrow C_*(K)$ such that the composition $C_*(K) \stackrel{C_*(f_M)}{\longrightarrow} C_*(M) \stackrel{C_*(r)}{\longrightarrow} C_*(K)$ is the identity. Since $C_i(f_M)$ is an isomorphism for $i \leq k$, we get that $C_k(r)$ is an isomorphism too, and this implies that the differential $C_{k+1}(M) \rightarrow C_k(M)$ vanishes. Therefore $C_*(M)$ splits.

\eqref{item:theorem:tau-pd-c} Let $\overline{m} = -m \colon M \rightarrow \R$ be the reverse Morse function on $M$. It has the same critical points as $m$, with index-$i$ critical points turning into index-$(n-i)$ critical points. It determines a new CW complex homeomorphic to $M$, which we will denote by $\overline{M}$. By cellular approximation there is a cellular map $\iota \colon M \rightarrow \overline{M}$ homotopic to $\id_M$. 

Now let $I \colon C^{n-*}(M) \rightarrow C_*(\overline{M})$ denote the isomorphism that sends the (cochain) dual of an $(n-i)$-cell of $M$ to the corresponding $i$-cell of $\overline{M}$. Then the chain homotopy equivalence $\PD \colon  C^{n-*}(M) \rightarrow C_*(M)$ inducing Poincar\'e duality can be defined (up to chain homotopy) as the composition $C_*(\iota)^{-1} \circ I$, where $C_*(\iota)^{-1}$ denotes the homotopy inverse of the chain homotopy equivalence $C_*(\iota) \colon C_*(M) \rightarrow C_*(\overline{M})$. 

The chain complex $C_*(\overline{M})$ splits, because $C^{n-*}(M)$ splits and $I$ is an isomorphism. By \cref{lem:split-cheq}, we have that $\PD$, $I$, and $C_*(\iota)$ all restrict to chain homotopy equivalences between the upper and lower halves of the chain complexes involved, and $C_*(\iota)^{-1} \big| _{C^{\ell}_*(\overline{M})} = (C_*(\iota) \big| _{C^{\ell}_*(M)})^{-1}$ and $C_*(\iota)^{-1} \big| _{C^u_*(\overline{M})} = (C_*(\iota) \big| _{C^u_*(M)})^{-1}$ (up to chain homotopy). Hence we have:
\[ \PD \big| _{C^u(M)^{n-*}} = (C_*(\iota) \big| _{C^{\ell}_*(M)})^{-1} \circ I \big| _{C^u(M)^{n-*}}, \quad \PD \big| _{C^{\ell}(M)^{n-*}} = (C_*(\iota) \big| _{C^u_*(M)})^{-1} \circ I \big| _{C^{\ell}(M)^{n-*}}.\] 
Since the isomorphism $I$ and its restrictions preserve the standard bases, this implies that 
\[ \tau(\PD \big| _{C^u(M)^{n-*}}) = -\tau(C_*(\iota) \big| _{C^{\ell}_*(M)}) \text{ and }  \tau(\PD \big| _{C^{\ell}(M)^{n-*}}) = -\tau(C_*(\iota) \big| _{C^u_*(M)}).\] 
Now note that $\tau(C_*(\iota) \big| _{C^u_*(M)}) + \tau(C_*(\iota) \big| _{C^{\ell}_*(M)}) = \tau(C_*(\iota)) = 0$. Here we used that $\iota$ is homotopic to $\id_M$, so we can apply~\cref{prop:chain-hom-chain-equivs-same-torsion} and \cref{lem:WT-ch-comp} to deduce that $\tau(C_*(\iota)) = 0$.  Therefore $\tau(\PD \big| _{C^{\ell}(M)^{n-*}}) = -\tau(C_*(\iota) \big| _{C^u_*(M)}) = \tau(C_*(\iota) \big| _{C^{\ell}_*(M)})$.

Let $L$ denote the $\left\lfloor \frac{n}{2} \right\rfloor$-skeleton of $\overline{M}$ and let $C = M \setminus h \circ i_1(\interior T)$. In $\overline{M}$ the handles corresponding to the critical points of $\overline{m}$ of index at most $\left\lfloor \frac{n}{2} \right\rfloor$ (which are the critical points of $m$ of index at least $\left\lceil \frac{n}{2} \right\rceil$) together make up $C$, and $L$ consists of the cores of these handles. By \cite[Theorem 4.18]{Ma02}, it follows that $C$ is homeomorphic to the mapping cylinder of the projection $\partial C \rightarrow L$, so the inclusion $L \rightarrow C$ is a simple homotopy equivalence by \cite[Corollary 5.1A]{Co73}. The cellular map $\iota$ restricts to a map $\iota \big| _{f_M(K)} \colon f_M(K) \rightarrow L$ between the $\left\lfloor \frac{n}{2} \right\rfloor$-skeletons of $M$ and $\overline{M}$. The inclusions determine isomorphisms $C^{\ell}_*(M) \cong C_*(f_M(K))$ and $C^{\ell}_*(\overline{M}) \cong C_*(L)$ (preserving the standard bases), hence $\tau(C_*(\iota) \big| _{C^{\ell}_*(M)}) = \tau(\iota \big| _{f_M(K)})$.  

First assume that (SP1) or (SP2) holds. Let $\varphi'$ denote the composition $K \stackrel{f_{\partial T}}{\longrightarrow} \partial T \stackrel{i_1}{\longrightarrow} i_1(\partial T) \rightarrow C$, so that $\tau(M,\varphi) = \tau(\varphi')$ (when $\pi_1(C)$ is identified with $\pi_1(M)$ via the inclusion). The composition $K \stackrel{f_M}{\longrightarrow} f_M(K) \stackrel{\iota}{\longrightarrow} L \rightarrow C \rightarrow M$ is homotopic to $\varphi$, because $f_M \simeq \varphi$ and $\iota \simeq \id_M$. So by \cref{lem:compl-he}~\eqref{item:lem-compl-he-b} the composition $K \stackrel{f_M}{\longrightarrow} f_M(K) \stackrel{\iota}{\longrightarrow} L \rightarrow C$ is homotopic to $\varphi'$. This shows that $\tau(\iota \big| _{f_M(K)}) = \tau(\varphi') = \tau(M,\varphi)$, because the homeomorphism $f_M \colon K \rightarrow f_M(K)$ and the inclusion $L \rightarrow C$ have vanishing Whitehead torsion. 

Now assume that (SP3) holds. Then $L = h \circ i_2 \circ f_T(K)$ and $r \big| _L \colon L \rightarrow K$ is a homeomorphism (by the definition of $r$). Moreover, $r \circ \iota \circ f_M \simeq r \circ \id_M \circ f_M = \id_K$. Since $f_M$ is also a homeomorphism, this means that $\iota \big| _{f_M(K)}$ is homotopic to a homeomorphism, so $\tau(\iota \big| _{f_M(K)}) = 0 = \tau(M,\varphi)$. 

Therefore $\tau(\PD \big| _{C^{\ell}(M)^{n-*}}) = \tau(M,\varphi)$ in all cases.
\end{proof}

\section{The Whitehead torsion of homotopy equivalences of doubles} \label{s:doubles-main}

In this section we prove \cref{theorem:SP-manifolds-and-tau-intro}. 
Fix positive integers $n, k$ with $n \geq 2k+1$. Let $M$ and $N$ be closed $n$-manifolds, and let $K$ and $L$ be finite CW complexes of dimension (at most) $k$. Suppose that $\varphi \colon K \rightarrow M$ and $\psi \colon L \rightarrow N$ are continuous maps such that $(M,\varphi)$ and $(N,\psi)$ are SP manifolds.

Let $F := \pi_1(M)$ and $G := \pi_1(N)$. The maps $\varphi$ and $\psi$ are $\left\lfloor \frac{n}{2} \right\rfloor$-connected by \cref{prop:double-conn}, so $\pi_1(\varphi)$ and $\pi_1(\psi)$ are isomorphisms. We use these isomorphisms to identify $\pi_1(K)$ with $F = \pi_1(M)$ and $\pi_1(L)$ with $G = \pi_1(N)$.

Let $w_M \colon F \rightarrow \left\{ \pm 1 \right\}$ and $w = w_N \colon G \rightarrow \left\{ \pm 1 \right\}$ be the orientation characters of $M$ and $N$ respectively. Fix twisted fundamental classes $[M] \in H_n(M;\Z^{w_M})$ and $[N] \in H_n(N;\Z^w)$. For any homotopy equivalence $f \colon M \rightarrow N$ we have $w \circ \pi_1(f) = w_M$ and $f_*([M]) = \varepsilon [N]$ for $\varepsilon = 1$ or $-1$.

\begin{definition}
For topological spaces $X$, $Y$, let $\hEq(X,Y)$ denote the set of homotopy classes of homotopy equivalences $X \to Y$. For an isomorphism $\theta \colon \pi_1(X) \rightarrow \pi_1(Y)$, let $\hEq_{\theta}(X,Y) \subseteq \hEq(X,Y)$ denote the subset consisting of homotopy equivalences $f$ such that $\pi_1(f) = \theta$.

For $\varepsilon = \pm 1$, let $\hEq(M,N)_{\varepsilon} \subseteq \hEq(M,N)$ denote the subset of degree-$\varepsilon$ homotopy equivalences, i.e.\ those that send $[M]$ to $\varepsilon [N]$. For an isomorphism $\theta \colon F \rightarrow G$, let $\hEq_{\theta}(M,N)_{\varepsilon} = \hEq_{\theta}(M,N) \cap \hEq(M,N)_{\varepsilon}$.
\end{definition}

\begin{lemma} \label{lem:hom-restr}
Suppose that $X$ is a CW complex of dimension at most $k$, and let $f,g \colon X \rightarrow L$ be continuous maps. If $\psi \circ f \simeq \psi \circ g \colon X \rightarrow N$, then $f \simeq g$. 
\end{lemma}

\begin{proof}
By \cref{theorem:tau-pd}~\eqref{item:theorem:tau-pd-a}, $N$ has a CW decomposition such that (up to homotopy) $\psi$ is an embedding identifying $L$ with the $\left\lfloor \frac{n}{2} \right\rfloor$-skeleton of $N$. 

If $(N,\psi)$ satisfies (SP1) or (SP2), then $N$ has no $(k+1)$-cells. Therefore if we make the homotopy between $\psi \circ f$ and $\psi \circ g$ cellular, we obtain a homotopy between $f$ and $g$. 

If $(N,\psi)$ satisfies (SP3), then we can compose the homotopy between $\psi \circ f$ and $\psi \circ g$ with the retraction $r \colon N \rightarrow L$ from the proof of \cref{theorem:tau-pd}~\eqref{item:theorem:tau-pd-b} to get a homotopy between $f$ and $g$. 
\end{proof}

\begin{lemma} \label{lem:heq-restr}
There is a well-defined restriction map $\hEq(M,N) \rightarrow \hEq(K,L)$. 
\end{lemma}

\begin{proof}
Again we fix CW decompositions on $M$ and $N$ using \cref{theorem:tau-pd}~\eqref{item:theorem:tau-pd-a}.

Consider a continuous map $M \rightarrow N$. After cellular approximation it can be restricted to a map $K \rightarrow L$, and by \cref{lem:hom-restr} the restriction's homotopy class is independent of the choice of the approximation. Therefore restriction defines a map $[M,N] \rightarrow [K,L]$. Similarly we get a map $[N,M] \rightarrow [L,K]$.

Now suppose that $f \colon M \rightarrow N$ is a cellular homotopy equivalence and $g$ is its cellular homotopy inverse. Then $f \circ g \simeq \id_N$ and $g \circ f \simeq \id_M$, hence $f \big| _{\varphi(K)} \circ g \big| _{\psi(L)} \simeq \id_{\psi(L)} \colon \psi(L) \rightarrow N$ and $g \big| _{\psi(L)} \circ f \big| _{\varphi(K)} \simeq \id_{\varphi(K)} \colon \varphi(K) \rightarrow M$. \cref{lem:hom-restr} implies that $f \big| _{\varphi(K)} \circ g \big| _{\psi(L)} \simeq \id_{\psi(L)} \colon \psi(L) \rightarrow \psi(L)$ and $g \big| _{\psi(L)} \circ f \big| _{\varphi(K)} \simeq \id_{\varphi(K)} \colon \varphi(K) \rightarrow \varphi(K)$, therefore $f \big| _{\varphi(K)} \colon \varphi(K) \rightarrow \psi(L)$ is a homotopy equivalence. 
\end{proof}

\begin{remark} \label{rem:heq-restr}
The restriction $\alpha \in \hEq(K,L)$ of a map $f \in \hEq(M,N)$ is characterised by the property that $\psi \circ \alpha \simeq f \circ \varphi \colon K \rightarrow N$, i.e.\ the following diagram is homotopy commutative: 
\[
\xymatrix{
M \ar[r]^-{f} & N \\
K \ar[r]^-{\alpha} \ar[u]^{\varphi} & L \ar[u]_{\psi}
}
\]
This implies that $\pi_1(\alpha) = \pi_1(f) \in \Hom(F,G)$. Therefore for every isomorphism $\theta \colon F \rightarrow G$ the restriction map of \cref{lem:heq-restr} restricts to a map $\hEq_{\theta}(M,N) \rightarrow \hEq_{\theta}(K,L)$. 
\end{remark}

Finally, we will establish the following, which is an equivalent formulation of \cref{theorem:SP-manifolds-and-tau-intro}.

\begin{theorem} \label{theorem:heq-diag}
For every isomorphism $\theta \colon F \rightarrow G$ with $w \circ \theta = w_M$ there is a commutative diagram
\[
\hspace{35mm}
\xymatrix{
\hEq_{\theta}(M,N) \ar[r]^-{\tau} \ar[d] & \Wh(G,w) \\
\hEq_{\theta}(K,L) \ar[r]^-{\tau} & \Wh(G,w) \ar[u]_{x \mapsto x-(-1)^n\overline{x}+\tau(N,\psi)-\theta_*(\tau(M,\varphi))}
}
\]
where the vertical map on the left is given by restriction.
\end{theorem}

\begin{proof}
Fix CW decompositions on $M$ and $N$ as in \cref{theorem:tau-pd}. We denote the corresponding (split) cellular chain complexes by $C_*(M) = C_*(M;\Z F)$ and $C_*(N) = C_*(N;\Z G)$. Let $\PD^M \colon C^{n-*}(M) \rightarrow C_*(M)$ and $\PD^N \colon C^{n-*}(N) \rightarrow C_*(N)$ denote the chain homotopy equivalences given by Poincar\'e duality. By \cref{theorem:tau-pd}~\eqref{item:theorem:tau-pd-c}, we have $\tau(\PD^M \big| _{C^{\ell}(M)^{n-*}}) = \tau(M,\varphi)$ and $\tau(\PD^N \big| _{C^{\ell}(N)^{n-*}}) = \tau(N,\psi)$. 
We will show that there is a commutative diagram 
\begin{equation}\label{diagram:heq-Wh}
\xymatrix{
\hEq_{\theta}(M,N)_1 \ar[r]^-{C_*({-})} \ar[d] & \chEq(C_*(M)_{\theta^{-1}},C_*(N))_{\PD^M,\PD^N} \ar[r]^-{\tau} \ar[d] & \Wh(G,w) \\
\hEq_{\theta}(K,L) \ar[r]^-{C_*({-})} & \chEq(C_*(K)_{\theta^{-1}},C_*(L)) \ar[r]^-{\tau} & \Wh(G,w). \ar[u]_{x \mapsto x-(-1)^n\overline{x}+\tau(N,\psi)-\theta_*(\tau(M,\varphi))}
}
\end{equation}
For the notation see  \cref{def:twist-module,def:ch-deg1}.  
We begin by describing the three maps not yet defined.

Since $K$ and $L$ are the $\left\lfloor \frac{n}{2} \right\rfloor$-skeletons of $M$ and $N$ respectively, we have $C^{\ell}_*(M) \cong C_*(K)$ (hence $C^{\ell}_*(M)_{\theta^{-1}} \cong C_*(K)_{\theta^{-1}}$) and $C^{\ell}_*(N) \cong C_*(L)$. By \cref{lem:split-cheq}, there is a restriction map $\chEq(C_*(M)_{\theta^{-1}},C_*(N)) \to \chEq(C_*(K)_{\theta^{-1}},C_*(L))$, and we define the vertical map in the middle to be its restriction to $\chEq(C_*(M)_{\theta^{-1}},C_*(N))_{\PD^M,\PD^N}$. 

Suppose that $f \in \hEq_{\theta}(K,L)$, then $f$ induces a chain map $C_*(K;\Z G^{\theta}) \rightarrow C_*(L;\Z G)$. We have $C_*(K;\Z G^{\theta}) \cong C_*(K;\Z F)_{\theta^{-1}}$ (see \cite[Lemma 2.26 (a)]{paper1}), so taking the induced chain map defines a map $C_*({-}) \colon \hEq_{\theta}(K,L) \rightarrow \chEq(C_*(K)_{\theta^{-1}},C_*(L))$. 

Similarly, a homotopy equivalence $f \in \hEq_{\theta}(M,N)$ 
induces maps 
\begin{align*} 
C_*(f) &\colon C_*(M;\Z G^{\theta}) \cong C_*(M)_{\theta^{-1}} \to C_*(N;\Z G) = C_*(N) \\ 
C^*(f) &\colon C^*(N;\Z G) = C^*(N) \to C^*(M;\Z G^{\theta}) \cong C^*(M)_{\theta^{-1}}
\end{align*}
where the last isomorphism follows from \cite[Lemma 2.26 (b)]{paper1}. Moreover, $C^*(M;\Z G^{\theta})$ is the dual of $C_*(M)_{\theta^{-1}}$, and $C^*(f)$ is identified with the dual of $C_*(f)$. The chain homotopy equivalence $C^{n-*}(M;\Z G^{\theta}) \cong C^{n-*}(M)_{\theta^{-1}} \rightarrow C_*(M;\Z G^{\theta}) \cong C_*(M)_{\theta^{-1}}$ given by Poincar\'e duality (with $\Z G^{\theta}$ coefficients) is identified with $\PD^M$ under the identification $\Hom(C^{n-*}(M)_{\theta^{-1}}, C_*(M)_{\theta^{-1}}) = \Hom(C^{n-*}(M), C_*(M))$. If $f$ has degree $1$, i.e.\ $f_*([M]) = [N]$, then it induces a homotopy commutative diagram
\[
\xymatrix{
C^{n-*}(M)_{\theta^{-1}} \ar[d]_{\PD^M} & C^{n-*}(N) \ar[l]_-{C^{n-*}(f)} \ar[d]^{\PD^N} \\
C_*(M)_{\theta^{-1}} \ar[r]^-{C_*(f)} & C_*(N)
}
\]
since Poincar\'e duality can be defined by taking cap product with the fundamental class. Therefore we get a restricted map $C_*({-}) \colon \hEq_{\theta}(M,N)_1 \rightarrow \chEq(C_*(M)_{\theta^{-1}},C_*(N))_{\PD^M,\PD^N}$.

Next we verify that the diagram in \eqref{diagram:heq-Wh} commutes. The square on the left commutes since both downward pointing arrows are defined by restriction. The square on the right commutes by \cref{lem:cheq-diag}. Note that if $\PD^M \big| _{C^{\ell}(M)^{n-*}}$ is regarded as a chain homotopy equivalence $C^{\ell}(M)^{n-*}_{\theta^{-1}} \rightarrow C^u_*(M)_{\theta^{-1}}$ (instead of $C^{\ell}(M)^{n-*} \rightarrow C^u_*(M)$), then its Whitehead torsion is $\theta_*(\tau(M,\varphi))$ (instead of $\tau(M,\varphi)$), see \cite[Lemma 2.27]{paper1}.

The Whitehead torsion of a homotopy equivalence is defined as the Whitehead torsion of the induced chain homotopy equivalence, so from \eqref{diagram:heq-Wh} we get a commutative diagram 
\[
\hspace{35mm}
\xymatrix{
\hEq_{\theta}(M,N)_1 \ar[r]^-{\tau} \ar[d] & \Wh(G,w) \\
\hEq_{\theta}(K,L) \ar[r]^-{\tau} & \Wh(G,w) \ar[u]_{x \mapsto x-(-1)^n\overline{x}+\tau(N,\psi)-\theta_*(\tau(M,\varphi))}
}
\]

We can apply the same argument to $-N$ instead of $N$ (where $-N$ is the same manifold $N$ with the opposite twisted fundamental class $[-N] = -[N]$). Then we get a commutative diagram
\[
\hspace{35mm}
\xymatrix{
\hEq_{\theta}(M,N)_{-1} \ar[r]^-{\tau} \ar[d] & \Wh(G,w) \\
\hEq_{\theta}(K,L) \ar[r]^-{\tau} & \Wh(G,w) \ar[u]_{x \mapsto x-(-1)^n\overline{x}+\tau(N,\psi)-\theta_*(\tau(M,\varphi))}
}
\]
because $\hEq_{\theta}(M,-N)_1 = \hEq_{\theta}(M,N)_{-1}$ and $\tau(-N,\psi) = \tau(N,\psi)$ (because the definition of $\tau(N,\psi)$ does not depend on the choice of the twisted fundamental class).

Since $\hEq_{\theta}(M,N) = \hEq_{\theta}(M,N)_1 \sqcup \hEq_{\theta}(M,N)_{-1}$, we can combine the two diagrams to get the diagram in the statement.
\end{proof}

\section{Applications} \label{s:appl}

Now we consider some applications of \cref{theorem:SP-manifolds-and-tau-intro} and prove the results announced in \cref{ss:appl}.

\subsection{Simple doubles} \label{ss:s-doubles}

Let $(M,\varphi)$ be a polarised manifold such that $\tau(M,\varphi)$ is defined. We say that $(M,\varphi)$ is \emph{simple} if $\tau(M,\varphi)=0$. Recall from \cref{cor:tau-triv-double,cor:tau-0} that if $(M,\varphi)$ has a trivial double structure, then it is simple, and if $(M,\varphi)$ is simple, then it has a twisted double structure. We now state a consequence of \cref{theorem:heq-diag} in the special case when $(M,\varphi)$ is simple, and then use it to prove \cref{thmx:doubles2}.

\begin{theorem}\label{thm:simple-doubles-first-application}
Suppose that $M$ and $N$ are $n$-manifolds, $K$ and $L$ are CW complexes of dimension $($at most$)$ $k$, and $\varphi \colon K \to M$ and $\psi \colon L \to N$ are continuous maps such that $(M,\varphi)$ and $(N,\psi)$ are SP manifolds. Let $G = \pi_1(N)$ with orientation character $w \colon G \rightarrow \left\{ \pm 1 \right\}$, and identify $\pi_1(L)$ with $G$ via $\psi$. Suppose that $\tau(M,\varphi)=0$. Then there is a commutative diagram
\[
\hspace{20mm}
\xymatrix{
\hEq(M,N) \ar[r]^-{\tau} \ar[d] & \Wh(G,w) \\
\hEq(K,L) \ar[r]^-{\tau} & \Wh(G,w) \ar[u]_{x \mapsto x-(-1)^n\overline{x}+\tau(N,\psi)}
}
\]
where the vertical map on the left is given by restriction.
\end{theorem}

\begin{proof}
We have $\hEq(M,N) = \bigsqcup_{\theta} \hEq_{\theta}(M,N)$, where the union ranges over all isomorphisms $\theta \colon \pi_1(M) \rightarrow G$ with $w \circ \theta = w_M$, where $w_M$ is the orientation character of $M$. For each such $\theta$ we can apply \cref{theorem:heq-diag}, and since $\theta_*(\tau(M,\varphi))=0$, we can combine the resulting diagrams to get the diagram in the statement.
\end{proof}

The following is obtained by further specialising.  

\begin{theorem} \label{theorem:haut-simple}
Suppose that $M$ is an $n$-manifold, $K$ is a CW complex of dimension (at most) $k$, and $\varphi \colon K \to M$ is a continuous map such that $(M,\varphi)$ is an SP manifold. Let $G = \pi_1(M)$ with orientation character $w \colon G \rightarrow \left\{ \pm 1 \right\}$, and identify $\pi_1(K)$ with $G$ via $\varphi$. Suppose that $\tau(M,\varphi)=0$. Then there is a commutative diagram
\[
\hspace{5mm}
\xymatrix{
\hAut(M) \ar[r]^-{\tau} \ar[d] & \Wh(G,w) \\
\hAut(K) \ar[r]^-{\tau} & \Wh(G,w) \ar[u]_{x \mapsto x-(-1)^n\overline{x}}
}
\]
where the vertical map on the left is given by restriction. 
\end{theorem}

\begin{proof}
Apply \cref{thm:simple-doubles-first-application} to the case where $M=N$, $K=L$, and $\varphi=\psi$. 
\end{proof}

In particular, this implies that $\tau(g) \in \mathcal{I}_n(G,w)$ for every $g \in \hAut(M)$, hence $T(M) \subseteq \mathcal{I}_n(G,w)$ and $U(M) = \{ 0 \}$ (see \cref{ss:shms}). We use this to prove \cref{thmx:doubles2}. 

\begin{theorem}[cf.\ Hausmann {\cite[Sections 9--10]{hausmann-mwmdh}}]\label{thm:hausmann}
Let $n \geq 5$, let $G$ be a finitely presented group and let $w \colon G \rightarrow \left\{ \pm 1 \right\}$ be such that $\psi \colon L_{n+1}^h(\Z G, w) \rightarrow \wh{H}^{n+1}(C_2;\Wh(G,w))$ is nontrivial $($see \cref{ss:shms}$)$. 
Then there exists an $n$-manifold $M$ with fundamental group $G$ and orientation character $w$ such that $|\M^h_{s,\hCob}(M)|>1$.
\end{theorem}

\begin{proof}
Let $K$ be a finite $2$-dimensional CW complex with $\pi_1(K) \cong G$ (and we fix an isomorphism). Let $\nu$ be a stable vector bundle over $K$ with orientation character $w$. By \cref{lem:thick-class} and \cref{rem:thick-5} there is an $n$-dimensional thickening $f_T \colon K \to T$ such that $f_T^*(\nu_T) \cong \nu$. Let $M = T \cup_g T$ be a twisted double of $T$ such that $(M,\varphi)$ is an SP manifold and $\tau(M,\varphi)=0$, where $\varphi$ is the composition of $f_T$ and the inclusion $T \to M$ of the first component (e.g.\ let $g = \id_{\partial T}$). Then $\pi_1(M) \cong G$ with orientation character $w$. 

It follows from \cref{theorem:haut-simple} that $U(M) = \{ 0 \}$. So $\image(\psi) \setminus U(M)$ is nonempty, and by \cref{prop:TU-cond} \eqref{prop:TU-cond-b} this implies that $|\M^h_{s,\hCob}(M)| > 1$. 
\end{proof}

\subsection{Doubles over manifolds} \label{ss:doubles-man}

We now consider the special case of \cref{theorem:heq-diag} when $K$ and $L$ are closed $k$-manifolds. We get an especially nice statement when $n-k$ is odd. Then we present some applications, proving Theorems \ref{thmx:sb} and \ref{thmx:doubles1}.

\begin{theorem} \label{theorem:haut-double-man}
Suppose that $M$ and $N$ are $n$-manifolds, $K$ and $L$ are $k$-manifolds and $\varphi \colon K \to M$ and $\psi \colon L \to N$ are continuous maps such that $(M,\varphi)$ and $(N,\psi)$ are SP manifolds. Let $F = \pi_1(M)$ and $G = \pi_1(N)$, and identify $\pi_1(K)$ with $F$ and $\pi_1(L)$ with $G$ via $\varphi$ and $\psi$ respectively. Let $w_M \colon F \rightarrow \left\{ \pm 1 \right\}$ be the orientation character of $M$ and let $w = w_N \colon G \rightarrow \left\{ \pm 1 \right\}$ be the orientation character of $N$. Assume that the orientation character of $L$ is also $w$. 
\benum 
\item\label{item:theorem:haut-double-man-a} If $n-k$ is odd, then $\tau(f)=\tau(N,\psi)-f_*(\tau(M,\varphi))$ for every homotopy equivalence $f \colon M \to N$. 

\item\label{item:theorem:haut-double-man-b} If $n-k$ is even, then for every isomorphism $\theta \colon F \rightarrow G$ with $w \circ \theta = w_M$ there is a commutative diagram
\[
\hspace{20mm}
\xymatrix{
\hEq_{\theta}(M,N) \ar[r]^-{\tau} \ar[d] & \Wh(G,w) \\
\hEq_{\theta}(K,L) \ar[r]^-{\tau} & \Wh(G,w) \ar[u]_{x \mapsto 2x+\tau(N,\psi)-\theta_*(\tau(M,\varphi))}
}
\]
\eenum
\end{theorem}

\begin{proof}
Let $f \colon M \to N$ be a homotopy equivalence. By \cref{lem:heq-restr} it restricts to a homotopy equivalence $\alpha \colon K \rightarrow L$. By \cref{theorem:heq-diag} we have $\tau(f) = \tau(\alpha)-(-1)^n\ol{\tau(\alpha)} +\tau(N,\psi)-f_*(\tau(M,\varphi)) \in \Wh(G,w)$. Since $K$ and $L$ are $k$-manifolds and the orientation character of $L$ is $w$, by \cref{prop:WT-man} we see that $\tau(\alpha) \in \mathcal{J}_k(G,w)$. That is, $\tau(\alpha)=-(-1)^k\ol{\tau(\alpha)} \in \Wh(G,w)$. Therefore \[\tau(f) = \tau(\alpha)-(-1)^n\ol{\tau(\alpha)} +\tau(N,\psi)-f_*(\tau(M,\varphi)) = \tau(\alpha)+(-1)^{n-k}\tau(\alpha) +\tau(N,\psi)-f_*(\tau(M,\varphi)).\qedhere\]
\end{proof}

We will use this to prove the following result from the introduction.

\begin{reptheorem}{thmx:sb}
Suppose that $j > k$ are positive integers and $j$ is odd. Let $K$ and $L$ be $k$-manifolds, and let $S^j \to M \to K$ and $S^j \to N \to L$ be orientable $($linear$)$ sphere bundles. Then every homotopy equivalence $f \colon M \to N$ is simple.
\end{reptheorem}

\begin{proof}
Let $n = j+k$ be the dimension of $M$ and $N$. It follows from the assumptions that $j \geq \max(3,k+1)$, so $n \geq \max(k+3,2k+1)$. 

The manifold $M$ is the sphere bundle of some orientable rank $(j+1)$ vector bundle $\xi$ over $K$. Since $j+1 > k$, there is a rank $j$ vector bundle $\xi_0$ such that $\xi \cong \xi_0 \oplus \varepsilon^1$, where $\varepsilon^1$ denotes the trivial rank $1$ bundle. Let $T$ be the disc bundle of $\xi_0$. Then $T$ is a thickening of $K$ (with the zero-section $K \rightarrow T$) and $M \approx T \cup_{\id_{\partial T}} T$. This means that if $\varphi$ denotes the composition $K \rightarrow T \rightarrow M$ of the zero section and the inclusion of the first component, then $(M,\varphi)$ has a trivial double structure, so it satisfies (SP3). Similarly, $(N,\psi)$ satisfies (SP3) for the analogous $\psi \colon L \rightarrow N$.

Let $G = \pi_1(N)$ and let $w \colon G \rightarrow \left\{ \pm 1 \right\}$ be the orientation character of $N$. Then $\pi_1(\psi)$ determines an identification $\pi_1(L) \cong G$, and since $\xi$ is orientable, $w$ is also the orientation character of $L$.

Since $M$ and $N$ are trivial doubles,  $\tau(M,\varphi) = 0$ and $\tau(N,\psi) = 0$ by \cref{cor:tau-triv-double}. Since $n-k=j$ is odd, \cref{theorem:haut-double-man}~\eqref{item:theorem:haut-double-man-a} implies that $\tau(f) = 0$ for every homotopy equivalence $f \colon M \to N$. 
\end{proof}

We obtain the following two theorems as immediate corollaries. 

\begin{theorem}
Suppose that $j > k$ are positive integers and $j$ is odd. Let $K$ and $L$ be $k$-manifolds, and let $S^j \to M \to K$ and $S^j \to N \to L$ be orientable sphere bundles. If $M$ and $N$ are homotopy equivalent, then they are simple homotopy equivalent.
\end{theorem}

\begin{theorem} \label{theorem:sb-tau-haut}
Suppose that $j > k$ are positive integers and $j$ is odd. Let $K$ be a $k$-manifold and let $S^j \to M \to K$ be an orientable sphere bundle. Then $T(M) = \{ 0 \}$ (see \cref{ss:shms}). 
\end{theorem}

\begin{remark}
In the three theorems above, the sphere bundle $M$ could be replaced with any twisted double $T \cup_g T$ of the disc bundle $T$ of an orientable rank $j$ bundle $\xi_0$ over $K$ such that $(M,\varphi)$ is an SP manifold and $\tau(M,\varphi)=0$, where $\varphi$ is the composition of the zero section $K \rightarrow T$ and the inclusion $T \to M$ of the first component (and similarly for $N$).
\end{remark}

Finally, we will use \cref{theorem:sb-tau-haut} to prove the following result from the introduction.

\begin{reptheorem}{thmx:doubles1}
Let $n \geq 11$ or $n=9$. Let $G$ be a finitely presented group with an orientation character $w \colon G \rightarrow \left\{ \pm 1 \right\}$. Then there is an $n$-manifold $M$ with fundamental group $G$ and orientation character $w$ such that $|\M^{\hCob}_s(M)| > 1$ if and only if $\mathcal{I}_n(G,w) \ne 0$. 
\end{reptheorem}

\begin{proof}
First assume that $\mathcal{I}_n(G,w) \ne 0$. Under the assumptions on $n$ there is an integer $k \geq 4$ such that $n-k > k$ and $n-k$ is odd. For instance, we can take $k=4$ for $n \geq 9$ odd and $k=5$ for $n \geq 12$ even.  Since $k \geq 4$, there is a $k$-manifold $K$ with $\pi_1(K) \cong G$ and orientation character $w$. Let $M$ be an orientable $S^{n-k}$-bundle over $K$ (e.g.\ $K \times S^{n-k}$). Then $\pi_1(M) \cong G$ with orientation character $w$, and by \cref{theorem:sb-tau-haut}, it follows that $T(M) = \{ 0 \}$. So $\mathcal{I}_n(G,w) \setminus T(M)$ is nonempty, and by \cref{prop:TU-cond} \eqref{prop:TU-cond-a} this implies that $|\M^{\hCob}_s(M)| > 1$. 

In the other direction, if $f$ is the homotopy equivalence induced by an $h$-cobordism between $n$-manifolds with fundamental group $G$ and orientation character $w$, then $\tau(f) \in \mathcal{I}_n(G,w)$ by \cref{prop:WT-hcob}. If the manifolds are not simple homotopy equivalent, then $\tau(f) \neq 0$, so $\mathcal{I}_n(G,w) \ne 0$.
\end{proof}

\subsection{Doubles over certain $2$-complexes} \label{ss:doubles-2c}

In this section we will consider specific $2$-dimensional CW complexes for which it is known by work of Metzler~\cite{Me79} that the Whitehead torsions of their homotopy automorphisms are contained in a certain subgroup of the Whitehead group. We can exploit this property by applying \cref{theorem:haut-simple} to doubles over such $2$-complexes, leading to a proof of \cref{thmx:metzler}.

Recall that for any group presentation $\mathcal{P} = \langle g_1, \ldots , g_s \mid r_1, \ldots , r_t \rangle$ there is a corresponding presentation complex, denoted by $X_{\mathcal{P}}$, which consists of one $0$-cell, one $1$-cell for each generator $g_i$, and one $2$-cell for each relation $r_j$, with gluing map determined by $r_j$.

In this section all groups will be equipped with the trivial orientation character $w \equiv 1$, which we will omit from the notation. 

We will need the Bass-Heller-Swan decomposition of the Whitehead group of $C_\infty \times C_m$: 
\begin{equation} \label{eq:BHS}
\Wh(C_\infty \times C_m) \xrightarrow{\cong} \Wh(C_m) \oplus \wt K_0(\Z C_m) \oplus NK_1(\Z C_m)^2.
\end{equation}
In \eqref{eq:BHS} each term is equipped with a natural involution, and the isomorphism respects these involutions (see \cite[p.329, p.357]{Ra86}). 

The following is a generalisation of the main result of \cite{Me79}.

\begin{theorem} \label{theorem:Metzler}
Let $m \geq 1$ and let $X := X_{\mathcal{P}}$, where $\mathcal{P} = \langle x, y \mid y^m, [x,y] \rangle$ is the standard presentation for $C_{\infty} \times C_m$. Then the composition
\[ \hAut(X) \xrightarrow[]{\tau} \Wh(C_{\infty} \times C_m) \twoheadrightarrow  
\wt K_0(\Z C_m) \]
is the zero map. In particular, $\tau(\hAut(X)) \subseteq \Wh(C_m) \oplus \{0\} \oplus NK_1(\Z C_m)^2 \subseteq \Wh(C_\infty \times C_m)$.
\end{theorem}

\begin{proof}
Let $G = C_{\infty} \times C_m$ and let  $\psi \colon \Z C_m \twoheadrightarrow \Z C_m/\Sigma$ be the quotient map, factoring out by the ideal generated by $\Sigma := \sum_{i=0}^{m-1} y^i$, the group norm in $\Z C_m$. Since $\Z G \cong (\Z C_m)[C_\infty]$ and $\Z G/\Sigma \cong (\Z C_m/\Sigma)[C_\infty]$, $\psi$ induces a map $\Psi \colon \Z G \twoheadrightarrow \Z G/\Sigma$. Metzler showed~\cite[Lemma 2]{Me79} that the composition 
\[ \hAut(X) \xrightarrow[]{\tau} \Wh(G) \cong K_1(\Z G)/{\pm G} \xrightarrow[]{\Psi_*} K_1(\Z G/\Sigma)/(\Z G/\Sigma)^\times \]
is the zero map. Since $\Z G/\Sigma \cong (\Z C_m/\Sigma)[t,t^{-1}]$, we set $R=  \Z C_m/\Sigma$, so that $\Z G/\Sigma \cong R[t,t^{-1}]$.  A variant of the Bass-Heller-Swan decomposition for arbitrary Laurent polynomial rings $R[t,t^{-1}]$~\cite[III.3.6]{We13} implies that
\[ K_1(\Z G/\Sigma)/(\Z G/\Sigma)^\times \cong (K_1(\Z C_m/\Sigma)/(\Z C_m/\Sigma)^\times) \oplus \wt K_0(\Z C_m/\Sigma) \oplus NK_1(\Z C_m/\Sigma)^2. \]
The map $\Psi_*$ respects this splitting and so we have that:
\begin{align*} 
\tau(\hAut(X)) \subseteq &\ker(\psi_* \colon \Wh(C_m) \to K_1(\Z G/\Sigma)/(\Z G/\Sigma)^\times) \\
&\quad \oplus \ker(\psi_* \colon \wt K_0(\Z C_m) \to \wt K_0(\Z C_m/\Sigma)) \oplus \ker(\psi_* \colon NK_1(\Z C_m) \to NK_1(\Z C_m/\Sigma))^2.
\end{align*}

It therefore suffices to prove that $\psi_* \colon \wt K_0(\Z C_m) \to \wt K_0(\Z C_m/\Sigma)$ is injective. To see this, consider the following pullback square of rings
\begin{equation}\label{eqn:Milnor-square}
\begin{tikzcd}
\Z C_m \ar[r,"\psi"] \ar[d,"\varepsilon"] & \Z C_m / \Sigma \ar[d,"\varepsilon'"] \\
\Z \ar[r,"\psi'"] & \Z/m
\end{tikzcd}
\end{equation} 
where $\psi, \psi'$ are the quotient maps and $\varepsilon,\varepsilon'$ are induced by augmentation. 
Note that \eqref{eqn:Milnor-square} has the property that at least one of the maps $\psi'$ and $\varepsilon'$ is surjective, i.e.\ \eqref{eqn:Milnor-square} is a Milnor square.
It follows from \cite[Theorem 42.13]{CR87} that \eqref{eqn:Milnor-square} induces a long exact sequence:
\[ K_1(\Z) \oplus K_1(\Z C_m/ \Sigma) \xrightarrow[]{(\psi'_*,\varepsilon'_*)} K_1(\Z/m) \xrightarrow[]{\partial} \wt K_0(\Z C_m) \xrightarrow[]{(\epsilon_*,\psi_*)} \wt K_0(\Z) \oplus \wt K_0(\Z C_m / \Sigma). \]
More specifically, by \cite[Theorem 42.13]{CR87}, we obtain an exact sequence with $\wt K_0$ replaced by $K_0$ throughout. It is clear from the definition of $\partial$ that its image lies in $\wt K_0(\Z C_m)$ and, from this, we obtain the exact sequence above.

By \cite[p.~343]{CR87}, we have $\im(\partial) = T(C_m) \subseteq \wt K_0(\Z C_m)$ where $T(G) \subseteq \wt K_0(\Z G)$ denotes the Swan subgroup of a finite group $G$. By \cite[Proposition 53.6~(iii)]{CR87}, we have $T(C_m) = 0$ and so $\partial = 0$.
Since $\Z$ is a PID, we have $\wt K_0(\Z) = 0$. Hence $\psi_* \colon \wt K_0(\Z C_m) \to \wt K_0(\Z C_m/\Sigma)$ is injective. It follows that the composition in the statement of the theorem is the zero map.
\end{proof}

This implies that the map $\tau\colon \hAut(X) \to \Wh(C_{\infty} \times C_m)$ is not surjective when $\wt K_0(\Z C_m) \ne 0$, which is a broad generalisation of \cite[Theorem 1]{Me79}.

\begin{corollary} \label{cor:metzler}
Let $n \geq 5$ and $m \geq 1$. Suppose that $M$ is an $n$-manifold and $\varphi \colon X_{\mathcal{P}} \rightarrow M$ is a continuous map such that $(M,\varphi)$ is an SP manifold and $\tau(M,\varphi)=0$, where $\mathcal{P} = \langle x, y \mid y^m, [x,y] \rangle$ is the standard presentation for $C_{\infty} \times C_m$. Then the composition
\[ 
\hAut(M) \xrightarrow[]{\tau} \Wh(C_{\infty} \times C_m) \twoheadrightarrow \wt K_0(\Z C_m)
\]
is the zero map. 
\end{corollary}

\begin{proof}
Consider the diagram
\[
\begin{tikzcd}
\hAut(M) \ar[r,"\tau"] \ar[d] & \Wh(C_{\infty} \times C_m) \ar[r,twoheadrightarrow] & \wt K_0(\Z C_m) \\
\hAut(X_{\mathcal{P}}) \ar[r,"\tau"] & \Wh(C_{\infty} \times C_m) \ar[r,twoheadrightarrow] \ar[u,"x \mapsto x-(-1)^n\overline{x}"'] & \wt K_0(\Z C_m) \ar[u,"x \mapsto x-(-1)^n\overline{x}"']
\end{tikzcd}
\]
where the vertical map on the left is given by restriction. The first square commutes by \cref{theorem:haut-simple}. The second square commutes, because the isomorphism \eqref{eq:BHS} is compatible with the involutions. Since the composition of the maps in the bottom row vanishes by \cref{theorem:Metzler}, the commutativity of the diagram implies that the composition of the maps in the top row vanishes too.
\end{proof}

\begin{remark}
By the first square of the diagram, we also have $T(M) \subseteq \mathcal{I}_n(C_{\infty} \times C_m)$. By \cite[Proposition 5.10]{paper1} $\mathcal{I}_n(C_{\infty} \times C_m) \cong \mathcal{I}_n(C_m) \oplus \{ x - (-1)^n\ol{x} \mid x \in \wt{K}_0(\Z C_m)\} \oplus NK_1(\Z C_m)$, so \cref{cor:metzler} implies that $T(M) \subseteq \mathcal{I}_n(C_m) \oplus \{0\} \oplus NK_1(\Z C_m)$, where $NK_1(\Z G)$ is embedded into $NK_1(\Z G)^2$ by the map $x \mapsto (x,-(-1)^n\ol{x})$. 
\end{remark}

\begin{theorem} \label{theorem:metzler-doubles}
Let $n \geq 5$ and let $m \ge 2$ be such that $\{x-(-1)^n \ol{x} \mid x \in \wt K_0(\Z C_m)\} \ne 0$. Then there is an orientable $n$-manifold $M$ with fundamental group $C_{\infty} \times C_m$ such that $|\M^{\hCob}_s(M)| > 1$.
\end{theorem}

\begin{proof}
Let $\mathcal{P} = \langle x, y \mid y^m, [x,y] \rangle$, and let $T$ be an oriented thickening of $X_{\mathcal{P}}$ (e.g.\ a regular neighbourhood of an embedding $X_{\mathcal{P}} \rightarrow \R^n$). Let $M = T \cup_g T$ be a twisted double of $T$ such that $(M,\varphi)$ is an SP manifold and $\tau(M,\varphi)=0$, where $\varphi \colon K \to M$ denotes the composition of $f_T$ and the inclusion $T \to M$ of the first component (e.g.\ let $g = \id_{\partial T}$). Then $\pi_1(M) \cong \pi_1(X_{\mathcal{P}}) \cong C_{\infty} \times C_m$ and $M$ is orientable. 

By \cref{prop:TU-cond} \eqref{prop:TU-cond-a} it is enough to show that $\mathcal{I}_n(C_{\infty} \times C_m) \setminus T(M)$ is nonempty. It follows from \cite[Proposition 5.10]{paper1} and \cref{cor:metzler} that $\mathcal{I}_n(C_{\infty} \times C_m) \setminus T(M)$ contains the subset $\{x-(-1)^n \ol{x} \mid x \in \wt K_0(\Z C_m)\} \setminus \{ 0 \}$, which is nonempty by our assumption.
\end{proof}

\subsection{The dependence of $\tau(M,\varphi)$ on $\varphi$} \label{ss:depend}

\cref{theorem:heq-diag} allows us to describe the set of possible values of $\tau(M,\varphi)$ for a fixed $M$.

\begin{proposition} \label{prop:tau-double-j}
Suppose that $n \geq \max(6,2k+2)$, $M$ is an $n$-manifold and $G = \pi_1(M)$ with orientation character $w \colon G \rightarrow \left\{ \pm 1 \right\}$. Let $\varphi \colon K \rightarrow M$ be an $\left\lfloor \frac{n}{2} \right\rfloor$-connected map for some CW complex $K$ of dimension $k$. Then $\tau(M,\varphi) \in \mathcal{J}_n(G,w)$.
\end{proposition}

\begin{proof}
By \cref{prop:double-conn} there is a generalised double structure $h \colon T \cup_{g_0} W \cup_{g_1} T \rightarrow M$ on $(M,\varphi)$. Let $i_1$, $i_2$, and $i_3$ denote the inclusions of the three components of $T \cup_{g_0} W \cup_{g_1} T$. Let $d_0 \colon \partial_1 W \rightarrow \partial_0 W$ be the composition of the inclusion $\partial_1 W \rightarrow W$ and the homotopy inverse of $\partial_0 W \rightarrow W$. By the construction of $d_0$, we have $i_2 \big| _{\partial_1 W} \simeq i_2 \big| _{\partial_0 W} \circ d_0 \colon \partial_1 W \rightarrow T \cup_{g_0} W \cup_{g_1} T$. Hence, with the notation $d \coloneqq g_0 \circ d_0 \circ g_1 \in \hAut(\partial T)$, and using that $g_1$ and $g_0$ are the gluing maps, we have
\begin{equation}\label{eqn:a-list-of-homotopies}
i_3 \big| _{\partial T} = i_2\big| _{\partial_1 W} \circ g_1 \simeq i_2\big| _{\partial_0 W} \circ d_0 \circ g_1 = i_1 \big| _{\partial T} \circ g_0 \circ d_0 \circ g_1 = i_1 \big| _{\partial T} \circ d \colon \partial T \rightarrow T \cup_{g_0} W \cup_{g_1} T.
\end{equation}
By \cref{prop:tau-double-comp} there is a homotopy automorphism $\alpha \colon K \rightarrow K$ such that 
\[
\varphi \circ \alpha \simeq h \circ i_3 \circ f_T
\]
and
\[
\tau(M,\varphi) = (h \circ i_2)_*(\tau(W,\partial_1 W)) - \varphi_*(\tau(\alpha)) .
\]

By \cref{rem:dt-sp}, $(\partial T, f_{\partial T})$ is an SP manifold, so by \cref{lem:heq-restr} there is a restriction map $\hAut(\partial T) \rightarrow \hAut(K)$. The homotopy automorphism $\alpha$ is the restriction of $d$. To see this recall \cref{rem:heq-restr} and note that 
\[ 
h \circ i_1 \circ f_{\partial T} \circ \alpha \simeq h \circ i_1 \circ f_T \circ \alpha \simeq \varphi \circ \alpha \simeq h \circ i_3 \circ f_T \simeq h \circ i_3 \circ f_{\partial T} \simeq h \circ i_1 \circ d \circ f_{\partial T} \colon K \rightarrow M, 
\] 
using \cref{def:doubles}, the definition of $\alpha$, and \eqref{eqn:a-list-of-homotopies}.
The map $h \circ i_1 \big| _{\partial T} \colon \partial T \rightarrow M$ is $\left\lfloor \frac{n}{2} \right\rfloor$-connected, because it is the composition of the inclusion $\partial T \rightarrow T$ (which is $\left\lfloor \frac{n}{2} \right\rfloor$-connected by \cref{lem:thick-conn}) and $h \circ i_1 \colon T \rightarrow M$ (which is also $\left\lfloor \frac{n}{2} \right\rfloor$-connected, because $h \circ i_1 \circ f_T \simeq \varphi$ and $f_T$ is a homotopy equivalence). Since $K$ has dimension $k \leq \left\lfloor \frac{n}{2} \right\rfloor -1$, we get that $f_{\partial T} \circ \alpha \simeq d \circ f_{\partial T} \colon K \rightarrow \partial T$, and this means that $\alpha$ is the restriction of $d$.

Since $(\partial T, f_{\partial T})$ has a trivial double structure, $\tau(\partial T, f_{\partial T})=0$. It follows from \cref{theorem:heq-diag} that $\tau(d) = (f_{\partial T})_*(\tau(\alpha) - (-1)^{n-1} \overline{\tau(\alpha)})$. Let $j = h \circ i_1 \big| _{\partial T} \colon \partial T \rightarrow M$. Then we have 
\begin{align*} 
j_*(\tau(d)) &= (h \circ i_1 \circ f_{\partial T})_*(\tau(\alpha) - (-1)^{n-1} \overline{\tau(\alpha)}) = (h \circ i_1 \circ f_T)_*(\tau(\alpha) - (-1)^{n-1} \overline{\tau(\alpha)}) \\
&= \varphi_*(\tau(\alpha) - (-1)^{n-1} \overline{\tau(\alpha)}) .
\end{align*}
Again we used that $h \circ i_1 \circ f_T \simeq \varphi$ (\cref{def:doubles}), to obtain the last equality. 

By \cref{prop:WT-hcob}, we have $\tau(d_0) = \tau(W,\partial_1 W) - (-1)^{n-1} \overline{\tau(W,\partial_1 W)}$ where $\pi_1(\partial_0 W)$ is identified with $\pi_1(W)$ via the inclusion. Since $g_0$ and $g_1$ are diffeomorphisms, we have $\tau(d) = (g_0)_*(\tau(d_0))$, by \cref{prop:WT-composition} and \cref{thm:chapman}. Hence 
\begin{align*}
j_*(\tau(d)) &= (h \circ i_1 \circ g_0)_*(\tau(d_0)) = (h \circ i_2)_*(\tau(d_0)) \\
&= (h \circ i_2)_*(\tau(W,\partial_1 W)) - (-1)^{n-1} \overline{(h \circ i_2)_*(\tau(W,\partial_1 W))}.
\end{align*}

By combining the above formulae we obtain 
\begin{align*}
\tau(M,&\varphi) - (-1)^{n-1} \overline{\tau(M,\varphi)} \\
&= (h \circ i_2)_*(\tau(W,\partial_1 W)) - (-1)^{n-1} \overline{(h \circ i_2)_*(\tau(W,\partial_1 W))} 
 - \varphi_*(\tau(\alpha)) + (-1)^{n-1} \overline{\varphi_*(\tau(\alpha)}) \\
&= j_*(\tau(d)) - j_*(\tau(d)) = 0 .
\end{align*} 
Therefore $\tau(M,\varphi) = - (-1)^n \overline{\tau(M,\varphi)}$, i.e.\ $\tau(M,\varphi) \in \mathcal{J}_n(G,w)$.
\end{proof}

If $n < \max(6,2k+2)$ and $(M,\varphi)$ is an SP manifold, then it satisfies (SP3), so $\tau(M,\varphi) = 0$. Therefore $\tau(M,\varphi) \in \mathcal{J}_n(G,w)$ for every SP manifold $(M,\varphi)$.

\begin{remark}
The proof of \cref{prop:tau-double-j} shows that $\tau(M,\varphi)$ can be regarded as a secondary invariant of an inertial $h$-cobordism on $\partial T$. An inertial $h$-cobordism on $\partial T$ is an $h$-cobordism between two copies of $\partial T$, more precisely, an $h$-cobordism $W$ with $\partial W = \partial_0 W \sqcup \partial_1 W$ together with diffeomorphisms $g_0 \colon \partial_0 W \rightarrow \partial T$ and $g_1 \colon \partial T \rightarrow \partial_1 W$. If $d \colon \partial T \rightarrow \partial T$ is the homotopy equivalence induced by $W$, $g_0$ and $g_1$, then $\tau(d) \in \mathcal{I}_{n-1}(G,w)$ for two different reasons: $\tau(d) = \tau(W,\partial_1 W) - (-1)^{n-1} \overline{\tau(W,\partial_1 W)}$ by \cref{prop:WT-hcob}, and $\tau(d) = \tau(\alpha) - (-1)^{n-1} \overline{\tau(\alpha)}$ by \cref{theorem:heq-diag}, where $\alpha \in \hAut(K)$ is the restriction of $d$. In general $\tau(W,\partial_1 W) \neq \tau(\alpha)$ (the former does not depend on the diffeomorphisms $g_0$ and $g_1$, but $\alpha$ does), and $\tau(W,\partial_1 W) - \tau(\alpha) = \tau(M,\varphi)$ for $M = T \cup_{g_0} W \cup_{g_1} T$ and $\varphi = i_1 \circ f_T$.  So $\tau(M,\varphi)$ is a secondary invariant in the  sense that it equals the difference between two reasons that $\tau(d) \in \mathcal{I}_{n-1}(G,w)$, i.e.\ two cochains in the Tate cochain group $\widehat{C}^{n-1}(C_2;\Wh(G,w))$  mapping to~$\tau(d)$ under the coboundary map. 
\end{remark}

\begin{proposition} \label{prop:tau-double-i}
Suppose that $M$ is an $n$-manifold, $G = \pi_1(M)$ and $w \colon G \rightarrow \left\{ \pm 1 \right\}$ is the orientation character of $M$. Let $\varphi \colon K \rightarrow M$ and $\varphi' \colon K' \rightarrow M$ be continuous maps such that $(M,\varphi)$ and $(M,\varphi')$ are SP manifolds. Then $\tau(M,\varphi) - \tau(M,\varphi') \in \mathcal{I}_n(G,w)$.
\end{proposition}

\begin{proof}
We will apply \cref{theorem:heq-diag} to $\id_M$, regarded as a map between the SP manifolds $(M,\varphi)$ and $(M,\varphi')$. Then $\id_M \in \hEq_{\id}(M,M)$, and it has a restriction $f \in \hEq_{\id}(K,K')$, where $\pi_1(K)$ and $\pi_1(K')$ are identified with $G$ via $\varphi$ and $\varphi'$. Let $x = \tau(f)$. Since $\id_M$ is a diffeomorphism, $\tau(\id_M) = 0$. Therefore $0 = x-(-1)^n\overline{x}+\tau(M,\varphi')-\tau(M,\varphi)$. This means that 
\[\tau(M,\varphi) - \tau(M,\varphi') = x-(-1)^n\overline{x} \in \mathcal{I}_n(G,w).\qedhere\] 
\end{proof}

\begin{remark}
Suppose that $(M,\varphi)$ is an SP manifold satisfying (SP1) and $k \geq 3$. Let $G = \pi_1(M) \cong \pi_1(K)$, and let $w \colon G \rightarrow \left\{ \pm 1 \right\}$ be the orientation character of $M$. Since $k \geq 3$, for every $x \in \Wh(G)$ there is a $k$-dimensional CW complex $K'$ and a homotopy equivalence $f \colon K' \rightarrow K$ such that $\tau(f) = x$. Let $\varphi' = \varphi \circ f \colon K' \rightarrow M$, then by \cref{prop:double-conn} $\varphi$, and hence $\varphi'$, is $\left\lfloor \frac{n}{2} \right\rfloor$-connected, so $(M,\varphi')$ also satisfies (SP1). The proof of \cref{prop:tau-double-i} shows that $\tau(M,\varphi') - \tau(M,\varphi) = x-(-1)^n\overline{x}$. Therefore if $\psi$ varies across all maps $L \rightarrow M$ such that $(M,\psi)$ is an SP manifold, then the set of possible values of $\tau(M,\psi)$ is precisely the coset $\tau(M,\varphi) + \mathcal{I}_n(G,w)$ in $\mathcal{J}_n(G,w)$. 
\end{remark}

\subsection{An invariant of unpolarised manifolds} \label{ss:tau}

Based on the observations of \cref{ss:depend}, we define an invariant for manifolds that can be obtained from SP manifolds by forgetting the polarisation. First we describe these manifolds.

\begin{definition}
Let $M$ be an $n$-manifold. 
\begin{compactitem}[$\bullet$]
\item We say that $M$ is a \emph{split manifold} if there is a positive integer $k$, a $k$-dimensional CW-complex $K$ and a continuous map $\varphi \colon K \rightarrow M$ such that $(M, \varphi)$ is an SP manifold.
\item Suppose that $n\geq 7$. We say that $M$ is \emph{strongly split} if it has a CW-decomposition with no $\left\lfloor \frac{n}{2} \right\rfloor$-cells.
\end{compactitem}
\end{definition}

\begin{proposition} \label{prop:split-strong}
Suppose that $n \geq 7$ and $M$ is an $n$-manifold. Then $M$ is strongly split if and only if there is a positive integer $k$, a $k$-dimensional CW-complex $K$ and a continuous map $\varphi \colon K \rightarrow M$ such that $(M, \varphi)$ satisfies (SP1). In particular, if $M$ is strongly split, then it is split. 
\end{proposition}

\begin{proof}
First assume that $M$ is strongly split and fix a CW-decomposition on $M$ with no $\left\lfloor \frac{n}{2} \right\rfloor$-cells. Let $k = \left\lfloor \frac{n}{2} \right\rfloor - 1$, and let $K$ be the $k$-skeleton of $M$, then the inclusion $\varphi \colon K \rightarrow M$ is $\left\lfloor \frac{n}{2} \right\rfloor$-connected. By \cref{prop:double-conn} $(M, \varphi)$ has a generalised double structure, so it satisfies (SP1). 

Now assume that $(M, \varphi)$ satisfies (SP1) for some $k$, $K$ and $\varphi$. Then the CW-decomposition constructed in the proof of \cref{theorem:tau-pd} has no $\left\lfloor \frac{n}{2} \right\rfloor$-cells.
\end{proof}

\begin{corollary} \label{cor:ssplit-he}
Suppose that $n \geq 7$. If $M$ is a strongly split $n$-manifold and $N$ is homotopy equivalent to $M$, then $N$ is also strongly split.
\end{corollary}

\begin{proof}
By \cref{prop:split-strong} there is a $k$, a $k$-dimensional $K$ and a $\varphi \colon K \rightarrow M$ such that $n \geq 2k+2$ and $(M,\varphi)$ has a generalised double structure. By \cref{cor:gds-he} there is a $\psi \colon K \rightarrow N$ such that $(N,\psi)$ has a generalised double structure, so it also satisfies (SP1). Again by \cref{prop:split-strong}, $N$ is strongly split.
\end{proof}

\begin{corollary}
Let $M$ be an $n$-manifold. Then $M$ is split if and only if one of the following holds:
\begin{compactitem}[$\bullet$]
\item $n \geq 8$ is even and $M$ is strongly split.
\item $n \geq 7$ is odd and $M$ is strongly split or $(M, \varphi)$ has a trivial double structure for some $K$ of dimension $\frac{n-1}{2}$ and $\varphi \colon K \rightarrow M$.
\item $n = 6$ and $(M, \varphi)$ has a twisted double structure for some $K$ of dimension at most $2$ and $\varphi \colon K \rightarrow M$.
\item $5 \geq n \geq k+3$ and $(M, \varphi)$ has a trivial double structure for some positive integer $k$, some $K$ of dimension at most $k$ and $\varphi \colon K \rightarrow M$.
\end{compactitem}
\end{corollary}

\begin{proof}
If $n \geq 7$, then (SP2) implies (SP1). If $n \geq 6$ is even, then (SP3) implies (SP2). 
\end{proof}

Next we define an invariant for split manifolds.

\begin{definition} \label{def:tau-unpol}
Suppose that $M$ is a split $n$-manifold, and let $G = \pi_1(M)$ with orientation character $w \colon G \rightarrow \left\{ \pm 1 \right\}$. We define the \emph{$\tau$-invariant of $M$} by 
\[
\tau(M) = \pi(\tau(M,\varphi)) \in \wh{H}^{n+1}(C_2;\Wh(G,w))
\]
where $k$ is a positive integer, $K$ is a $k$-dimensional CW complex and $\varphi \colon K \rightarrow M$ is a continuous map such that $(M,\varphi)$ is an SP manifold.
\end{definition}

It follows from \cref{prop:tau-double-j,prop:tau-double-i} that $\tau(M)$ is well-defined. 

By the proof of \cref{theorem:tau-pd}, if $n \geq 8$ and $M$ is strongly split, then it is a manifold without middle dimensional handles in the sense of Hausmann \cite{hausmann-mwmdh}. For such manifolds $\tau(M)$ coincides with the ``torsion invariant" defined in \cite[Section 9]{hausmann-mwmdh}. Moreover, Hausmann showed that this invariant is preserved by simple homotopy equivalences and homotopy equivalences induced by $h$-cobordisms. Both of the two theorems below can be regarded as strengthenings of this statement.

Firstly, in analogy with \cref{theorem:heq-diag}, we obtain the following formula for $\pi(\tau(f))$ for a homotopy equivalence $f$ between split manifolds.

\begin{theorem} \label{theorem:tau-diff}
Suppose that $M$ and $N$ are split $n$-manifolds and $f \colon M \rightarrow N$ is a homotopy equivalence. Let $G = \pi_1(N)$ with orientation character $w \colon G \rightarrow \left\{ \pm 1 \right\}$. Then 
\[
\pi(\tau(f)) = \tau(N) - f_*(\tau(M)) \in \wh{H}^{n+1}(C_2;\Wh(G,w)).
\]
In particular, if $f \colon M \rightarrow N$ is a homotopy equivalence between split manifolds, then $\pi(\tau(f))$ depends only on the induced homomorphism $\pi_1(f)$. 
\end{theorem}

\begin{proof}
By \cref{prop:WT-man} $\tau(f) \in \mathcal{J}_n(G,w)$, so $\pi(\tau(f))$ is defined. Fix a $\varphi \colon K \rightarrow M$ and a $\psi \colon L \rightarrow N$ such that $(M,\varphi)$ and $(N,\psi)$ are SP manifolds. Let $\alpha \colon K \rightarrow L$ be the restriction of $f$, and let $x = \tau(\alpha) \in \Wh(G)$. By \cref{theorem:heq-diag} $\tau(f) = x-(-1)^n\overline{x}+\tau(N,\psi)-f_*(\tau(M,\varphi))$. Therefore 
\begin{align*}
\pi(\tau(f)) &= \pi(x-(-1)^n\overline{x}) + \pi(\tau(N,\psi)) - \pi(f_*(\tau(M,\varphi))) \\
&= 0 + \pi(\tau(N,\psi)) - f_*(\pi(\tau(M,\varphi))) = \tau(N) - f_*(\tau(M)).\qedhere
\end{align*}
\end{proof}

Secondly, we show that $\tau$ is a complete invariant for the equivalence relation generated by simple homotopy equivalence and $h$-cobordism, restricted to split manifolds (cf.\ \cite[Proposition 4.15]{paper1}).

\begin{theorem} \label{theorem:tau-complete}
Suppose that $M$ and $N$ are split $n$-manifolds. The following are equivalent.  
\begin{clist}{(a)}
\item\label{item:theorem-tau-complete-a} There is a homotopy equivalence $f \colon M \rightarrow N$ such that $\tau(N) = f_*(\tau(M))$.
\item\label{item:theorem-tau-complete-b} There is a manifold $P$ that is simple homotopy equivalent to $M$ and $h$-cobordant to $N$.
\end{clist}
\end{theorem}

\begin{proof}
If $n = 4$, then $(N,\psi)$ has a trivial double structure for some $1$-dimensional $L$ and $\psi \colon L \to N$, so $\tau(N)=0$, and similarly $\tau(M)=0$. We also get that $\pi_1(N) \cong \pi_1(L)$ is free. By Stallings \cite{St65} the Whitehead group of a finitely generated free group is trivial, so every homotopy equivalence $M \rightarrow N$ is simple. Hence each of \eqref{item:theorem-tau-complete-a} and \eqref{item:theorem-tau-complete-b} holds if and only if $M \simeq N$. In the rest we assume that $n \geq 5$. 

$\eqref{item:theorem-tau-complete-a} \Rightarrow \eqref{item:theorem-tau-complete-b}$. Let $G = \pi_1(N)$ with orientation character $w \colon G \rightarrow \left\{ \pm 1 \right\}$. By \cref{theorem:tau-diff}, we have $\pi(\tau(f)) = 0$, so $\tau(f) \in \mathcal{I}_n(G,w)$. By \cite[Corollary 3.3]{paper1} (using the assumption that $n \geq 5$) this implies that there exists an $n$-manifold $P$ that is simple homotopy equivalent to $M$ and $h$-cobordant to $N$.

$\eqref{item:theorem-tau-complete-b} \Rightarrow \eqref{item:theorem-tau-complete-a}$. Let $f = h \circ g$ for a simple homotopy equivalence $g \colon M \rightarrow P$ and a homotopy equivalence $h \colon P \rightarrow N$ induced by an $h$-cobordism. Then $\tau(f) = \tau(h) + h_*(\tau(g))$. Since $\tau(g)=0$, this implies that $\tau(f) \in \mathcal{I}_n(G,w)$ by \cref{prop:WT-hcob}. Therefore $\pi(\tau(f)) = 0$, so $\tau(N) = f_*(\tau(M))$ by \cref{theorem:tau-diff}.
\end{proof}

\addtocontents{toc}{\protect\addvspace{2em}}

\def\MR#1{}
\bibliography{biblio}
\end{document}